\documentclass[11pt]{amsart}
\usepackage{amsmath,amsfonts,amsthm,amssymb,amscd,hyperref,latexsym,xy,enumerate}

\DeclareFontFamily{OT1}{rsfs}{}
\DeclareFontShape{OT1}{rsfs}{n}{it}{<-> rsfs10}{}
\DeclareMathAlphabet{\mathscr}{OT1}{rsfs}{n}{it}
\DeclareMathOperator{\id}{id}
\DeclareMathOperator{\BS}{BS}
\DeclareMathOperator{\Sym}{Sym}
\DeclareMathOperator{\mo}{\,mod}
\xyoption{all}
\theoremstyle{definition}                                                    
\newtheorem{theorem}{Theorem}
\newtheorem{lemma}[theorem]{Lemma}
\newtheorem{proposition}[theorem]{Proposition}
\newtheorem{corollary}[theorem]{Corollary}
\newtheorem{definition}{Definition}

\newtheorem*{remark}{Remark}

\numberwithin{equation}{section}
\usepackage{color}
\begin{document}
\title{Soficity, short cycles and the Higman group}
\author{Harald A. Helfgott}
\address{Harald A. Helfgott,
Mathematisches Institut,
Georg-August Universit\"{a}t G\"{o}ttingen, Bunsenstra{\ss}e 3-5, D-37073 G\"{o}ttingen, Germany; IMJ-PRG, UMR 7586,
 58 avenue de France, B\^{a}timent S. Germain, case 7012,
 75013 Paris CEDEX 13, France}
\email{helfgott@math.univ-paris-diderot.fr}
\author{Kate Juschenko} 
\address{Kate Juschenko, Department of Mathematics,
Northwestern University, 2033 Sheridan Road, Evanston, IL 60208, USA}
\email{kate.juschenko@gmail.com}

\begin{abstract}
  This is a paper with two aims. First, we show that the map from
  $\mathbb{Z}/p\mathbb{Z}$ to itself defined by exponentiation $x\to m^x$
  has few $3$-cycles -- that is to say, the number of cycles of length three is $o(p)$. 
This improves on previous bounds.

Our second objective is to contribute to an ongoing discussion
on how to find a non-sofic group. In
  particular, we show that, if the Higman group were sofic, there would be
  a map from $\mathbb{Z}/p\mathbb{Z}$ to itself, locally like an
  exponential map, yet satisfying a recurrence property.



\end{abstract}

\maketitle

\section{Introduction}\label{sec:intro}
The questions treated in this paper are motivated in part by
the study of {\em sofic groups}, and, more specifically, by the search
for a non-sofic group. (Non-sofic groups are not, as of the moment of
writing, yet known to exist. See the survey \cite{MR2460675}.) The same 
questions have been studied from
a different angle in coding theory; vd. \cite{MR2644246} and references
therein. 

As we shall discuss, it is natural to test whether the {\em Higman group} is
sofic. We shall show that, if the Higman group were sofic, then there would have
to exist a function $f:\mathbb{Z}/p\mathbb{Z}\to \mathbb{Z}/p\mathbb{Z}$ that
is, to say the least, odd-looking: it would behave locally like an exponential
map almost everywhere,
but $f\circ f\circ f\circ f$ would be equal to the identity.

We do not succeed here in showing that such a function does not exist.
However, we do prove another kind of
result. Just from the triviality of an analogue
of the Higman group, we obtain easily that there is no map $f$ that behaves
locally like an exponential map and has $f\circ f\circ f$ equal to the identity
almost everywhere. If $f$ is actually an exponential map, i.e.,
if $f$ is given by $f(x) = m^x \mo p$ for $x=0,1,\dotsc,p-1$ and some
$m\not\equiv 0,1 \mo p$, then
we can actually prove
that $f\circ f\circ f$ is {\em different} from the identity almost everywhere;
in other words, $f$ has few $3$-cycles. (Here ``few'' means ``$o(x)$''.) 
This improves on the best result known on
$3$-cycles to prime moduli, namely, \cite[Thm. 6]{MR2644246}.
(This is a matter of independent interest; it has been studied before, and
apparently, never in relation to sofic groups.)

Returning to the function $f$ that has to exist if the Higman group is sofic:
our strategy for extracting the existence of such a function from soficity
involves the restriction of a so-called
{\em sofic approximation} to
{\em amenable} subgroups of the Higman group.
All sofic approximations of an
amenable group are, in an asymptotic sense, conjugate; this enables us to
``rectify'' them, i.e., conjugate them so as to take them to sofic approximations with geometric or arithmetic meaning.

In this way, we show that, if the Higman group is sofic, then there is a map
$f$ that behaves locally like an exponential map almost everywhere
and satisfies $f^4=e$. Does this in fact suggest that the Higman group is not sofic?  It would be premature to venture a definite answer.
If $f$ behaves like an exponential map with too few exceptions, then we do arrive at a contradiction by $p$-adic arguments. This can be seen as a hint in the direction of non-soficity.

At the same time, there are two provisos. Shortly after the first draft of the present paper appeared, Glebsky \cite{Gleb} showed that analogous maps $f:\mathbb{Z}/p^n \mathbb{Z} \to \mathbb{Z}/p^n \mathbb{Z}$ do exist when $p|m-1$,
where $m$ is the base of the exponential.\footnote{In the original version of our work,
we had $m=2$; then, of course, $m-1$ has no prime divisors $p$.}
 Such maps come from the fact that generalizations $H_{4,m}$ of the Higman group $H_4$ have finite $p$-quotients. The Higman group $H_4$ itself has no finite quotients, so the construction
 does not immediately apply to the function $f$ arising from it.


Secondly, very recent work by Kuperberg, Kassabov and Riley \cite{kkr} -- motivated in part by the present paper -- shows in a different way that functions
with counterintuitive properties similar to those of our function $f$ exist. The proof has some elements in common with ours; the existence of a function resembling $f$ follows from the existence of a sofic group (not Higman's)
with certain properties. 
We will discuss this matter in \S \ref{sec:heur}.
 The question of the existence of our function $f$ remains open;
so does the soficity of the Higman group.

\subsection{Main results}

Let us start with the simple result on $3$-cycles we just mentioned. Its proof
uses almost no machinery.

\begin{definition}\label{defn:fg}
  Let $m,n \geq 2$ be coprime integers. We define
  $f_{m,n}:\{0,1,\dotsc,n-1\}\to \{0,1,\dotsc,n-1\}$ to be
 the map defined by
\[f_{m,n}(x) \equiv m^x \mo n.\]
\end{definition}

\begin{theorem}\label{thm:mathov}
  Let $m,n \geq 2$ be coprime integers.
  Then \[f_{m,n}(f_{m,n}(f_{m,n}(x)))=x\] can hold for at most $o_m(n)$ elements
$x\in \{0,1,\dotsc,n-1\}$.
\end{theorem}
Here, as usual, $o_m(n)$ means ``a function $b_m(n)$ such that 
$\lim_{n\to\infty} b_m(n)/n = 0$'', where the subscript $m$ warns us that
the function $b_m(n)$ may depend on $m$. We also use the notation
$O_m(n)$, meaning ``a function $B_m(n)$ such that 
$B_m(n)/n$ is bounded''. As is standard, we also use $o(n)$ or $O(n)$ without
subscripts when dependencies are non-existent or obvious.

Compare Theorem \ref{thm:mathov}
to \cite[Thm. 6]{MR2644246}, which states that, for $p$ prime, the number of
$x$ for which $f_{m,p}(f_{m,p}(f_{m,p}(x)))=x$ is $\leq 3p/4 + O_m(1)$. (The number of such 
$x$ is of interest in part because of the study of $f_{m,p}$ in the context of
the generation of pseudorandom numbers:
see the references \cite{MR1670959}, \cite{MR1982971}
given in \cite{MR2644246}.)

While the results analogous to Theorem \ref{thm:mathov} with $f_{m,n}(x)$ or
$f_{m,n}(f_{m,n}(x))$ instead of $f_{m,n}(f_{m,n}(f_{m,n}(x)))$ are very easy,
the problem with more than $3$ iterations is hard if $n$ is a prime. Indeed,
for $n$ prime and $m$ such that $m\mo n$ generates
$(\mathbb{Z}/n\mathbb{Z})^*$,
we do not know how to prove that
\begin{equation}\label{eq:hutor}f_{m,n}(f_{m,n}(f_{m,n}(f_{m,n}(x)))) = x\end{equation}
  can hold for at most $o(n)$ elements $x$ of $\{0,1,\dotsc,n-1\}$, or even that
  it can hold for at most $n-1$ elements of $\{0,1,\dotsc,n-1\}$.
  
On the other hand, we do have good upper bounds (namely, $< p^k$)
on the number of solutions
to
\begin{equation}f_{m,n}(f_{m,n}(\dotsb(f_{m,n}(x)))) = x
\;\;\;\;\;\;\;\;\;\;\;\;\text{($k$ times)}\end{equation}
when $n=p^r$, $r\geq k+1$
\cite[Cor. 3]{MR3118384}, \cite[Thm. 5.7]{MR2999154}.
The argument in  \cite{MR2999154} is based on $p$-adic analysis, whereas 
\cite{MR3118384} is based on explicit matrix computations -- though
it arguably still has Hensel's lemma at its core.

\begin{center}
  * * *
\end{center}

There turns out to be a relation between counting solutions to
(\ref{eq:hutor}) and an important problem in asymptotic group theory,
namely, that of constructing a {\em non-sofic group}.
If (\ref{eq:hutor}) (with $m=2$) could hold for $(1-o(1)) n$ elements of
$\{0,1,\dotsc,n-1\}$, $n$ odd, then a non-trivial quotient of
the {\em Higman group}
would be sofic, as one can easily see from the definitions (which we are about
to give). More interestingly, as we are about to see,
if the Higman group were in fact sofic, then
there would be a map $f$, locally like $f_{2,n}$, such that
\[f(f(f(f(x)))) = x\]
would hold for $(1-o(1)) n$ elements of $\{0,1,\dotsc,n-1\}$ (where we can
take $n$ to be a prime, or even a power $p^r$, $r\geq 5$, say).

Let us recall some standard notation. We write $\Sym(n)$ for the symmetric group, i.e.,
the group of all permutations of a set with $n$ elements.
The (normalized) {\em Hamming distance} $d_h(g_1,g_2)$
between two permutations
$g_1, g_2\in \Sym(n)$ is defined to be the number of elements at which
they differ, divided by $n$:
\[d_h(g_1,g_2) = \frac{1}{n} |\{1\leq i\leq n: g_1(i)\ne g_2(i)\}|,\]
where we write $|S|$ for the number of elements of a set $S$.
It is clear that $d_h$ is an (left- and right-) invariant metric on $\Sym(n)$.
Write $\id$ for the identity element of $\Sym(n)$.

\begin{definition}\label{def:sofic}
  Let $G$ be a group. For $n\in \mathbb{Z}^+$,
  $\delta>0$ and $S\subset G$ a finite subset,
  an {\em $(S,\delta,n)$-sofic approximation} is a map
  $\phi:S\to \Sym(n)$ such that
\begin{enumerate}
\item\label{it:amarg} 
 $d_h(\phi(g) \phi(h),\phi(g h)) <\delta$
  for all $g, h\in S$ such that $g h\in S$ (``$\phi$ is an approximate
  homomorphism''),
\item\label{it:omorg} $d_h(\phi(g),\id)>1-\delta$ for all $g\in S$ distinct from the identity $e$ (i.e., the image of every $g\ne e$ has few fixed points).
\end{enumerate}

We say that the group $G$ is {\em sofic} if, for every finite subset $S\subset G$ and every $\delta>0$, there is an $(S,\delta,n)$-sofic approximation
for some $n\in \mathbb{Z}^+$.
\end{definition}
It is easy to see that (\ref{it:amarg}) implies that $d_h(\phi(e),\id)<\delta$
for the identity $e\in G$, provided that $e\in S$.

The notion of sofic groups goes back to the work of Gromov, who used a different, but equivalent, definition. See \cite{MR2460675} for a survey. It is clear that,
if a group is sofic, all of its subgroups are sofic as well. It is also immediate
that, if a group is {\em not} sofic, then it has a finitely generated (and,
in particular, countable) subgroup that is not sofic either.

\begin{definition}
  Let $n,m\geq 2$. We denote by $H_{n,m}$ the group generated by
  elements $a_1, a_2,\dotsc, a_n$ 
subject to the following relations:
\[a_i^{-1} a_{i+1} a_i = a_{i+1}^m\;\;\;\;\text{(for $1\leq i<n$)},
\;\;\;\;\;\;\;\;\;
a_n^{-1} a_1 a_n = a_1^m.\]
The group $H_4 = H_{4,2}$ is called the {\em Higman group}.
\end{definition}
The group $H_{2,2}$ is trivial (this is easy). The group $H_{3,2}$ is trivial as well;
this is shown at the end of \cite{MR0038348},
where the proof is credited to K. A. Hirsch.

The Higman group was first constructed as an example of a group
without finite quotients \cite{MR0038348}. It is not known whether it has
amenable quotients. (We will go over the concept of {\em amenability} in
\S \ref{sec:amensof}; amenability implies soficity.\footnote{On the other hand,
 there is
  a sofic group for which there is no sequence of amenable groups converging
  to it in the Chabauty topology \cite{MR2794910}.})
Because of this, as well as for other reasons
(see \cite{MR2900231}), the Higman group $H_4 =
H_{4,2}$ can be seen as a plausible candidate for a non-sofic group.

What is more -- $H_4$ is {\em SQ-universal} \cite{MR0291298}, meaning that
every countable group is isomorphic to some subgroup of some quotient of
$H_4$. This implies immediately that, if a non-sofic group exists, then
there exists a non-sofic quotient $H_4/N$ of $H_4$.

\begin{theorem}\label{thm:jut}
  Let $m\geq 2$.
  Assume that the group $H_{4,m}$ is sofic. Then, for every $\epsilon>0$,
there is an $N>0$ such that, for every $n\geq N$ coprime
to $m$, there is a bijection
$f:\mathbb{Z}/n\mathbb{Z} \to \mathbb{Z}/n\mathbb{Z}$ such that
\begin{equation}\label{eq:cdr}f(x+1) = m f(x)\end{equation}
for at least $(1-\epsilon) n$ elements $x$ of
$\mathbb{Z}/n\mathbb{Z}$, and
\begin{equation}\label{eq:crossbolt}
f(f(f(f(x)))) = x\end{equation}
for all $x\in \mathbb{Z}/n\mathbb{Z}$.
\end{theorem}
This result has a very easy almost-converse: if, for $\epsilon>0$ arbitrary,
there are $n$ and $f:\mathbb{Z}/n\mathbb{Z}\to \mathbb{Z}/n\mathbb{Z}$  such that  (\ref{eq:cdr}) holds for at least $(1-\epsilon) n$ elements of
$\mathbb{Z}/n\mathbb{Z}$, then the Higman group has arbitrarily large
sofic quotients. (That is: it has either an infinite sofic
quotient, or arbitrarily large finite quotients.) In fact, this
is precisely what happens with $H_{4,m}$, $m>2$, as Glebsky \cite{Gleb} first pointed
out, and as we will discuss later.

If we assume that a given quotient $H_{4,m}/N$, $m$ arbitrary, is sofic,
then the proof
of Theorem 2 can be modified trivially to give not just the same conclusion,
but a stronger one, including equalities other than (\ref{eq:cdr}) and
(\ref{eq:crossbolt}). The same holds, in general, whenever we study a group $G$
into which the Baumslag-Solitar group\footnote{Baumslag-Solitar groups
  are defined in (\ref{eq:harmo}).} $\BS(1,m)$ embeds, whether or not this group $G$ is
a quotient of a group $H_{4,m}$; the proof of Theorem 
\ref{thm:jut} can be easily modified to show that the assumption that $G$ is
sofic implies that there is a bijection
$f:\mathbb{Z}/n\mathbb{Z} \to \mathbb{Z}/n\mathbb{Z}$ such that
(\ref{eq:cdr}) holds together with some other conditions in the place of
(\ref{eq:crossbolt}).
We will discuss this at the end of \S \ref{sec:lindur}.


We can choose to focus on $n$ prime, or, instead,
on $n$ a high power of a fixed prime, since then the
statement goes in the opposite direction to the results from
\cite[Cor. 3]{MR3118384}, \cite[Thm. 5.7]{MR2999154} we have mentioned, in
the sense that the consequence of soficity it points out is the negation
of a hypothetical stronger form of such results.
We will use the same $p$-adic tools as
in \cite{MR2999154} to prove the following statement.
\begin{proposition}\label{prop:norvi}
Let $m>1$, $p\nmid (m-1)$.
  Let a bijection $f:\mathbb{Z}/p^r \mathbb{Z} \mapsto \mathbb{Z}/p^r
  \mathbb{Z}$, $r\geq 1$ be given. Then either
\[f(f(f(f(x)))) \ne x\]
for at least $p^r/2$ values of $x\in \mathbb{Z}/p^r \mathbb{Z}$, or
  \begin{equation}\label{eq:anto}
    f(x+1) \ne m f(x)\end{equation}
  for at least $p^{r/4-1}/2^{1/4}$ 
values of $x\in \mathbb{Z}/p^r \mathbb{Z}$.
\end{proposition}
Of course, this is unfortunately much too weak to contradict
the conclusion of Theorem \ref{thm:jut}; for that, we would need
(\ref{eq:cdr}) to hold for $> p^r - p^{r/4-1}/\sqrt{2}$ values of $x$, not
just for $(1-\epsilon) p^r$ values of $x$. As we now know,
the conclusion of Theorem \ref{thm:jut} actually holds for $m>2$.

\begin{center}
  * * *
\end{center}

It is easy to give a probabilistic argument (\S \ref{sec:heur})
that the existence of a function $f$ such as that given by Theorem \ref{thm:jut}
is implausible (for $\epsilon<1/4$), and hence that it is also implausible
that the Higman group be sofic. Let us emphasize that this argument is merely
heuristic; it is very far from a proof.
The arithmetical flavor of Theorem \ref{thm:jut} might seem
to lend some credence
to the heuristics,
in so far as an assumption of independence of certain kinds of random variables 
underlies both the argument here and several classical
conjectures in number theory. However, we are not in a context that is fully
familiar to a number theorist, in particular, due to the large number of
``exceptions to the rule'' (namely, $\epsilon n$).

Moreover -- and this is important -- the work of both Glebsky \cite{Gleb} and
Kuperberg, Kassabov and Riley \cite{kkr} seems to cast doubt on these
heuristics. Glebsky's example assumes $n=p^k$, $p|(m-1)$. However,
by the argument at the beginning of the proof of Theorem \ref{thm:jut}
(\S \ref{sec:lindur}), it implies the existence of a function just like that 
in the conclusion of Theorem \ref{thm:jut} for every $m>2$ and
for every $n$ larger than a constant depending only on $m$ and $\epsilon$.

The idea behind Glebsky's construction may be summarized as follows.
If $n=p^k$ and $p|(m-1)$, then, evidently,  $m^x \equiv 1 \mo p$;
moreover, $\mathbb{Z}/p^k\mathbb{Z}$ has a tower of
subgroups $G_l=\mathbb{Z}/p^l \mathbb{Z}$ such that $x\to m x$ acts
on $G_l/G_{l+1}$ as the identity. It makes sense 
that iterations of such a map would tend to have fixed points.

The maps in \cite{kkr} are also of the type in Theorem \ref{thm:jut},
though there $m$ is not constant, but grows slowly with $n$.
As \cite{kkr} shows, these maps, like those coming from Glebsky's work,
go against the probabilistic heuristics discussed here.
Such heuristics must thus be taken with extreme circumspection, to say
the very least.

Returning to rigor: is it possible to give conclusions stronger than those
of Thm. \ref{thm:jut} if we make assumptions on ``how sofic'' $H_4$ is, i.e.,
assumptions on the dependence of $n$ on $S$ and $\delta$ in
Definition \ref{def:sofic}? (Such assumptions have been formalized in
different ways, as {\em sofic dimension growth} \cite{ArzhCher} and as
{\em sofic profile} \cite{MR3160544}.) This
question is related to that of strengthening the
methods we are about to discuss (\cite{MR2823074}, \cite{MR3068400}), or,
more generally, to the problem of giving versions of results on
{\em stability} and {\em weak-stability} \cite{MR3350728} with good bounds.
We will address these matters in \S \ref{sec:furthq}, but do not solve them.

To go back to Thm.\;\ref{thm:mathov}: a result resembling Thm.\;\ref{thm:mathov}, but with weaker conditions and  conclusions, follows easily
from the fact that $H_3$ is trivial. We will go over this at the beginning of
\S \ref{sec:fecy}. It would be interesting if the triviality of $H_3$ could
be used to prove Thm.\;\ref{thm:mathov} itself (or a statement with the same
conclusions but weaker conditions).
As we said, the proof of Thm.~\ref{thm:mathov} we give
requires next to no tools, though some will recognize the idea of Poincar\'e
recurrence at work.

\subsection{Methods}

The main tool used towards the proof of Theorem \ref{thm:jut} is the fact
that any two sofic representations of an amenable group are conjugate to each
other. What this means is that, if $G$ is an {\em amenable} group
(we shall recall the definition) generated by a finite set $S\subset G$,
and $\phi, \phi'$ are two $(S', \delta,n)$-sofic
approximations of $G$  with $S'\supset S$ large enough and $\delta$ small enough, then
there is a bijection $\tau$
from $\{1,\dotsc,n\}$ to itself such that, for every $s\in S$,
$\tau\circ \phi(s) \circ \tau^{-1}$
equals $\phi'(s)$ at almost all points
(i.e., the Hamming distance between $\tau\circ \phi(s) \circ \tau^{-1}$ and
$\phi'(s)$ is small).

While the group $H_{4,m}$ is not amenable, we can take $G$ to be an amenable subgroup of $H_{4,m}$
(to wit, the Baumslag-Solitar group $\BS(1,m)$). It is easy to see that
$G=\BS(1,m)$ has a sequence of sofic
approximations $\phi_k':G\to \Sym(n_k)$ with a
natural arithmetical description.
Hence, if $\{\phi_k\}$ is a sequence of sofic approximations of $H_{4,m}$,
the restriction of $\phi_k$ to $G$ must
be conjugate to $\phi_k'$. This constrains $\phi_k$ severely; the same sequence
of bijections $\tau'$ that show $\phi_k|_G$ to be conjugate to $\phi_k'$ shows that $\phi_k$
is conjugate to a sequence of maps having the properties given to $f$ in
Theorem \ref{thm:jut}.

In fact, we will be working with $(\mathbb{Z}/4\mathbb{Z})\ltimes H_{4,m}$ rather
than $H_{4,m}$, so as to strengthen the constraints on $\phi_k$. The reason
we can proceed in this way is that, 
if $H_{4,m}$ is sofic, then so is $(\mathbb{Z}/4\mathbb{Z})\ltimes H_{4,m}$,
since any extension of a sofic group by an amenable group
(such as $\mathbb{Z}/4\mathbb{Z}$) is sofic \cite{MR2220572}.

The fact that any two sofic representations of an amenable group are conjugate to
each other is something that has been stated and proved in different ways.
It was proved by Elek and Szabo \cite{MR2823074} using ``infinitary'' language
(ultraproducts, which depend on the axiom of choice).
We will use a ``finitary'' statement (based on a slight refinement
of \cite[Lem.~4.5]{MR3068400}) from which effective bounds could be easily
extracted.

{\bf Acknowledgements.} We would like to thank Sean Eberhard for suggesting
a simplification of the proof of Thm.~\ref{thm:mathov} and David Kerr for his remarks on
\cite{MR3068400} and \cite{MR3130315}. We are also grateful to
Goulnara Arzhantseva,
Ken Dykema, Lev Glebsky,
Martin Kassabov, Vivian Kuperberg, Alex Lubotzky, Tim Riley,
Andreas Thom and Nikolai Nikolov for stimulating discussions of soficity, to 
Igor Shparlinski for pointing out a gap in the first version of our preprint,
to Pablo Candela for a useful reference on generalizations of Rokhlin's 
Lemma, and to an anonymous referee.

Part of the work towards this paper was carried out
while H. Helfgott visited the Chebyshev Laboratory (St.~Petersburg)
and IMPA (Rio de Janeiro) and K. Juschenko visited the Weizmann Institute
(Rehovot) and the Bernoulli Centre (Lausanne). H. Helfgott is currently
supported by ERC Consolidator
grant GRANT and by funds from his Humboldt Professorship.
%

\section{Amenability and sofic approximations: tools and background}\label{sec:amensof}

Let $G$ be a group with a finite generating set $S$. In this case,
one of the (mutually equivalent) standard definitions of amenability
is as follows: $G$ is said to be
{\em amenable} if there exists
 an infinite sequence (called a {\em F\o lner sequence}) of increasing sets 
 $$e\in F_1\subset F_2\subset \ldots \subset F_k\subset \ldots
 \subset G$$
 such that $|s F_j\Delta F_j|\leq |F_j|/j \text{ for all } s \in S$ and all $j\geq 1$. 

 The aim of this section is to prove Prop.~\ref{prop:labil}, which states,
 in effect, that all sofic representations of an amenable group are conjugate.
 This is a key recent result that is neither new nor ours; nevertheless,
 we will have to give a proof, since we have not been able to find it
 in the literature in the concrete form we need (though its meaning is
 identical to that of \cite[Thm.~2]{MR2823074}, or rather the
 difficult direction thereof).
 
 The alternative would have been to derive
 Prop.~\ref{prop:labil} (a finitary statement) from
 \cite[Thm.~2]{MR2823074}, which uses ultraproducts in its proof and
 formulation.
   This would be much as in \cite{MR3350728}.

   \begin{center}
     * * *
   \end{center}
   
Now, given any $\eta>0$, we can actually assume that  
\begin{equation}\label{eq:huhuhu}
  |(F_{j-1}^{-1} F_{j}) \setminus F_{j}|\leq \eta |F_{j}|\end{equation}
  for all $j>1$, simply by replacing $\{F_j\}_{j\geq 1}$ by a subsequence, if necessary.
  
As it turns out, sofic approximations of amenable groups decompose
particularly nicely: any such approximation has an
almost-covering by trivial approximations of $F_j$.
Let us give a precise statement.
We take the following nomenclature from \cite{MR910005}: 
given $\epsilon\geq 0$ and a finite set $D$, we say
a collection $\{A_i\}_{i\in I}$ of subsets of $D$ is {\em $\epsilon$-disjoint}
if there exist pairwise disjoint subsets $A_i'\subset A_i$ such that
$|A_i'|\geq (1-\epsilon) |A_i|$, and that it
{\em $(1-\epsilon)$-covers} $D$ if $|\cup_{i\in I} A_i|\geq (1-\epsilon) |D|$.)

\begin{lemma}\label{lem:buend}
  Let $G$ be a group.
  For any $\epsilon, \kappa>0$, there
  are $k\geq 1$ and
  $\lambda_1,\dotsc,\lambda_k \in (0,1\rbrack$ with
  $1-\epsilon \leq \lambda_1+\dotsb+\lambda_k \leq 1$ such that
  the following holds.
  For any infinite sequence of finite subsets
  \[e\in F_1\subset F_2\subset \ldots \subset F_k\subset \ldots
  \subset G\] satisfying (\ref{eq:huhuhu}) for $\eta=\kappa/(24/\epsilon)^{k-1}$,
  there are $\delta>0$, $N\geq 1$
   and a finite set $S\subset G$ such that, if
  $\phi:S\to \Sym(n)$ is an $(S,\delta,n)$-sofic approximation with $n\geq N$,
    there exist $C_1,\dotsc,C_k\subset \{1,\dotsc,n\}$ such
  that
  \begin{enumerate}
  \item\label{conc:eins} the sets $\phi(F_1) C_1, \dotsc, \phi(F_k) C_k$ are pairwise disjoint,
  \item\label{conc:zwei} for every $1\leq j\leq k$ and every $c\in C_j$, the map
    $s\mapsto \phi(s) c$ from $F_j$ to $\phi(F_j) c \subset \{1,\dotsc,n\}$ is injective,
  \item\label{conc:drei} the family $\{\phi(F_j) c\}_{1\leq j\leq k,\; c\in C_j}$ is $\epsilon$-disjoint and $(1-\epsilon)$-covers $\{1,\dotsc,n\}$,
  \item\label{conc:vier} $(1-\kappa) \lambda_j \leq
    |\phi(F_j) C_j|/n \leq (1 + \kappa) \lambda_j$ for
    every $j=1,2,\dotsc,k$.
    \end{enumerate}
\end{lemma}
This is essentially  \cite[Lemma 4.5]{MR3068400}; we have only added
conclusion (\ref{conc:vier}), which will be crucial to our purposes. It was
already pointed out in
\cite[Lemma 4.3]{MR3130315} that the method of proof in 
\cite[Lemma 4.5]{MR3068400} can give conclusions like this one,
but the version of conclusion (\ref{conc:vier}) given there
is unfortunately
not quite strong for our purposes.

We remark in passing that the values of $\lambda_1,\dotsc,\lambda_k$ given
by the proof below depend only on $\epsilon$, not on $\kappa$, though we will
not need this in what follows.

\begin{proof}[Proof of Lemma \ref{lem:buend}]
We will start with an $(S,\delta,n)$-sofic approximation $\phi$ and show how to construct
the sets $C_1,\dotsc,C_k$, in reverse order. It will become clear that the argument works
provided that $k$ is larger than a constant depending only on
$\epsilon$ (not on $\kappa$).
We set $S = F_{k}^{-1} F_k$.

As in \cite{MR3068400}, we say that a collection $\{A_i\}$ of subsets of a finite set $X$ is a {\em
$\rho-$even covering of $X$ of multiplicity $M$} if (i) no element of $X$ is contained 
in more than $M$ elements of $\{A_i\}$, (ii) $\sum_i |A_i| \geq
(1-\rho) M |X|$.
Since $S = F_k^{-1} F_k$,
Definition \ref{def:sofic} implies that there is a set $B\subset \{1,2,\dotsc,n\}$
with $|B|\geq (1- \delta') n$, $\delta' = O_{|F_k|}(\delta)$,
such that, for all $g,h \in F_k$,
\begin{enumerate}
\item\label{it:cars}
    $\phi(g)^{-1} \phi(h) x = \phi(g^{-1} h) x$ for all $x\in B$,
\item\label{it:cbrs}
  $\phi(g^{-1} h) x \ne x$ for all $x\in B$ unless $g = h$.\end{enumerate}
By (\ref{it:cars}) and (\ref{it:cbrs}),
the map $s\mapsto \phi(s) x$ is
injective on $F_k$ for $x\in B$.
It is easy to see that this implies that
the sets $\phi(F_k) x$ form a $\delta'$-even
covering of $\{1,\dotsc,n\}$ of multiplicity $|F_k|$.

By \cite[Lemma 4.4]{MR3068400},
every $\rho$-even covering of a set $X$ contains a $\epsilon$-disjoint subcollection that $\epsilon (1-\rho)$-covers $X$. In our case, this means that
there is a set $C\subset B$ such that the sets $\phi(F_k) c$, $c\in C$, are $\epsilon$-disjoint and
satisfy $|\cup_{c\in C} \phi(F_k) c| \geq \epsilon (1-\delta') n$. We take $C_k$ to be a minimal such set
$C$, and set $\lambda_k = \epsilon$. Clearly, 
\[(1-\delta') \lambda_k n \leq |\phi(F_k) C_k| \leq (1-\delta') \lambda_k n + |F_k|.\]
Since $n\geq N$ and
we can assume that $N$ is larger than any given function of $\epsilon$,
$\kappa$ and $|F_k|$,
we may assume that $|F_k|$ is less than $n$ times an arbitrarily
small constant that we may let depend on $\epsilon$ and $\kappa$. Since we
can also take $\delta$ 
to be smaller than an arbitrary quantity depending on
$\epsilon$, $\kappa$ and $|F_k|$, we can assume that $\delta'$
is smaller than any given quantity depending on $\epsilon$ and $\kappa$,
we conclude that
\[(1-\kappa_k) \lambda_k n \leq |\phi(F_k) C_k| \leq (1+\kappa_k) \lambda_k n,\]
where $\kappa_k>0$ is an arbitrarily small constant
that we are allowed to
let depend on $k$, $\epsilon$ and $\kappa$; we will set it later.
(Of course, since $k$ will be determined by $\epsilon$ and $\kappa$,
a dependence on $k$, $\epsilon$ and $\kappa$ is the same as a dependence
on just $\epsilon$ and $\kappa$.)

We can now set up an iteration. For $j = k,k-1,\dotsc,1$, we will construct
sets $C_j\subset B$ such that
  \begin{enumerate}
  \item $\phi(F_j) C_j$ is disjoint from
    $\phi(F_{j+1}) C_{j+1} \cup \dotsc \cup \phi(F_k) C_k$,
  \item for every $c\in C_j$, the map
    $s\mapsto \phi(s) c$ defined on $F_j$ is injective,
  \item the family $\{\phi(F_j) c\}_{c\in C_j}$ is $\epsilon$-disjoint,
  \item $(1-\kappa_j) \lambda_j \leq
    |\phi(F_j) C_j|/n \leq (1 + \kappa_j) \lambda_j$ for
    every $j=1,2,\dotsc,k$, where $\kappa_j$
    will depend only on $\epsilon$, $\kappa$ and $j$,
    and, moreover, $\kappa_j \geq \kappa_{j+1}
    \geq \dotsc \geq \kappa_k$.
    \end{enumerate}
  We have just constructed $C_k$; let $B_k = B$.
  We shall now construct $C_j$, $j<k$, assuming we have already constructed
  $C_{j+1},\dotsc,C_k$. We define
  \[B_j = \{s\in B: \phi(F_j) s \cap (\phi(F_{j+1}) C_{j+1} \cup \dotsc \cup
  \phi(F_k) C_k) = \emptyset\}.\]
  By (\ref{eq:huhuhu}) (with $\eta=1$) and the fact that $C_i \subset B$ for all $i>j$,
  \[\begin{aligned}
  |\phi(F_j)^{-1} &(\phi(F_{j+1}) C_{j+1} \cup \dotsb \cup \phi(F_k) C_k)|\\
  &= |\phi(F_j^{-1} F_{j+1}) C_{j+1} \cup \dotsb \cup \phi(F_j^{-1} F_k) C_k|\\
  &\leq \sum_{i=j+1}^k (|\phi(F_j^{-1} F_i \setminus F_i) C_i| + |\phi(F_i) C_i|)\\
  &\leq \sum_{i=j+1}^k (\eta |F_i| |C_i| + |\phi(F_i) C_i|) \leq
  \sum_{i=j+1}^k \left(1 + \frac{\eta}{1-\epsilon}\right) |\phi(F_i) C_i|\\
  &\leq \left(1 + \frac{\eta}{1-\epsilon}\right) (1 + \kappa_{j+1})
  \sum_{i=j+1}^k \lambda_i n.\end{aligned}\]
  Writing
  \[\sigma_r = \sum_{i=r}^k \lambda_i, \;\;\;\;\;
  \delta_j = \delta' + \left(1 + \frac{\eta}{1-\epsilon}\right) (1 +
  \kappa_{j+1}) \sigma_{j+1},\]
  we obtain that
  \[|B_j| \geq |B| - \left(1 + \frac{\eta}{1-\epsilon}\right) (1 + \kappa_{j+1})
  \sigma_{j+1} n \geq (1 - \delta_j) n.\]

  Hence, $\{\phi(F_j) c\}_{c\in B_j}$ is a $\delta_j$-even covering of
  $\{1,2,\dotsc,n\}$ with multiplicity $|F_j|$.
  Just as before, we apply \cite[Lemma 4.4]{MR3068400}, and obtain a set
  $C_j\subset B_j$ such that the sets $\phi(F_j) c$, $c\in C_j$ are
  $\epsilon$-disjoint and
  \[\epsilon ( 1 -\delta_j) n \leq |\phi(F_j) C_j| \leq
  \epsilon ( 1 -\delta_j) n + |F_j|.\]
  Since $C_j \subset B_j$,
  the set $\phi(F_j) C_j$ is disjoint from the sets
  $\phi(F_{j+1}) C_{j+1},\dotsc, \phi(F_k) C_k$; since $B_j\subset B$,
  the map $s\mapsto \phi(s) c$ defined on $F_j$ is injective for all
  $c\in B_j$, and thus for all $c\in C_j$.

  Note now that
  \[\epsilon (1 - \delta_j) n = \epsilon \left(1 - 
\delta' - \left(1 + \frac{\eta}{1-\epsilon}\right) (1 +
\kappa_{j+1}) \sigma_{j+1}\right) n \geq \epsilon 
(1 - \sigma_{j+1}) (1 - \kappa_j) n,\]
where we set
\begin{equation}\label{eq:dubcross}
  \kappa_j = \max\left(\frac{\delta' +\sigma_{j+1}\left[ \kappa_{j+1} \left(1 +
      \frac{\eta}{1-\epsilon}\right) +
      \frac{\eta}{1-\epsilon}\right]}{1-\sigma_{j+1}}
  ,\kappa_{j+1}\right),
\end{equation}
note also that
\[\epsilon (1 - \delta_j) n + |F_j| \leq \epsilon (1 - \sigma_{j+1}) n\]
provided that \[n\geq \frac{|F_j|}{\epsilon(\delta_j - \sigma_{j+1})},\]
as we may assume. We let
\begin{equation}\label{eq:terix} \lambda_j = \epsilon (1 - \sigma_{j+1})
\end{equation} and conclude that
\begin{equation}\label{eq:martha}(1-\kappa_j) \lambda_j \leq
\frac{|\phi(F_j) C_j|}{n} \leq \lambda_j \leq (1 + \kappa_j) \lambda_j,\end{equation}
as desired. We see that we have proved the four conditions that we stated
$C_j$ would satisfy.

We continue this iteration until we reach $j=1$. Conclusions (\ref{conc:eins})
and (\ref{conc:zwei}) in the statement of the Lemma are immediate, as is the
first half of conclusion (\ref{conc:drei}). We will have conclusion
(\ref{conc:vier}) directly from (\ref{eq:martha}) provided that
\begin{equation}\label{eq:cross}
  \kappa_1 \leq \kappa.\end{equation}
Then
\[\left|\cup_{j=1}^k \phi(F_j) C_j\right| = \sum_{j=1}^k |\phi(F_j) C_j|
\geq n\cdot \sum_{j=1}^k (1 - \kappa_j) \lambda_j \geq n\cdot
\sum_{j=1}^k (1 - \kappa) \lambda_j.\]
Thus, the second half of conclusion (\ref{conc:drei}) holds (i.e.,
$\{\phi(F_j) C\}_{1\leq j\leq k, c\in C_j}$ does
$(1-\epsilon)$-cover $\{1,\dotsc,n\}$) provided that
\begin{equation}\label{eq:dot}
  \lambda_1 + \dotsc + \lambda_k \geq 1 - \epsilon/2\end{equation}
  and that 
  $\kappa \leq \epsilon/2$ (as we may assume: if it is not the case,
 we reset $\kappa = \epsilon/2$ at the very beginning).
Now, by (\ref{eq:terix}),
we have $(1 - \sigma_j) = (1 - \epsilon) (1 -
\sigma_{j+1})$, and so (\ref{eq:dot}) holds provided that $(1-\epsilon)^k
\leq \epsilon/2$; for that, in turn, to be true, it is enough to set
$k$ to be at least a constant times $\epsilon^{-1} \log \epsilon^{-1}$.
We can, in fact, choose $k$ such that
\begin{equation}\label{eq:tututu}
  1  - \epsilon/2 \leq \sigma_1 =
\lambda_1 + \dotsc + \lambda_k \leq 1 - \epsilon/4,\end{equation}
since (\ref{eq:terix}) holds and since we may assume $\epsilon<1/2$.

It remains to verify (\ref{eq:cross}). By (\ref{eq:dubcross}),
for $1\leq j\leq k-1$,
\[\kappa_j \leq \frac{1}{\epsilon/4} \left(
\delta' + 2 \eta + 3 \kappa_{j+1}\right).\]
Recall that $\delta' = O_{|F_k|}(\delta)$; we can set $\delta$ small
enough that $\delta'\leq \eta = \kappa/(24/\epsilon)^{k-1}$.
Recall also that may set
$\kappa_k$ to a very small value depending on $k$, $\epsilon$ and $\kappa$;
we set $\kappa_k = \eta$, and thus obtain that
\[\begin{aligned}
\kappa_{k-1} &\leq \frac{1}{\epsilon/4} \cdot 6 \eta =
\frac{\kappa}{(4/\epsilon)^{k-2}}\\
  \kappa_{k-2} &\leq  \frac{1}{\epsilon/4}
  (\delta' + 2 \eta + 3 \kappa_{k-1}) \leq
  \frac{1}{\epsilon/4} \cdot \frac{6 \kappa}{(4/\epsilon)^{k-2}} =
  \frac{\kappa}{(4/\epsilon)^{k-3}},\\
  \vdots\\
  \kappa_1 &\leq \kappa,\end{aligned}\]
where we assume, as we may, that $\epsilon\leq 1/4$.
Thus (\ref{eq:cross}) holds, and we are done.
\end{proof}

The fact that sofic approximations of an amenable group are (asymptotically)
conjugate is an easy corollary of 
Lemma \ref{lem:buend}, as the following shows. We thank D. Kerr for pointing
us in this direction.

\begin{proposition}\label{prop:labil}
  Let $G$ be an amenable group. Then, for any $\epsilon>0$ and any finite
  $S\subset G$,  there are a subset $S'\subset G$ with
  $S'\supset S$ and constants $N\in \mathbb{Z}^+$, $\delta>0$ such that,
  if $\phi_1$ and $\phi_2$ are $(S',\delta,n)$-sofic approximations of $G$
  with $n\geq N$, then there is a bijection $\tau:\{1,2,\dotsc,n\}\to
  \{1,2,\dotsc,n\}$ such that, for every $s\in S$,
  \[(\tau \circ \phi_1(s) \circ \tau^{-1})(y) = (\phi_2(s))(y)\]
  for $\geq (1-\epsilon) n$ values of $x\in \{1,2,\dotsc,n\}$.
\end{proposition}
\begin{remark}
  It is not actually necessary to require that $n\geq N$. The
  statement is both true and relatively straightforward for $n<N$:
  we can set $\delta\leq 1/N$, and then
  there are no $(S',\delta,n)$-sofic representations with
$n\leq N$, provided that $S'$ is large enough (why?); if $G$ is finite
(so that we could be kept from choosing $S'$ large enough), we set
  $S'=G$, thus making $\phi$ and $\phi'$ into injective homomorphisms from
  $G$ to $\Sym(n)$ whose
images are regular permutation subgroups. It is easy to show
then that any two such homomorphisms must be conjugate.
\end{remark}
\begin{proof}
  We will use Lemma \ref{lem:buend}.
  Its application requires a consecutive choice of $\epsilon$, $\kappa$ and
  a F\o lner sequence $\{F_{j}\}_{j\geq 1}$.
  We choose $\epsilon$ and $\kappa$ both equal to our $\epsilon/7$.
  Lemma \ref{lem:buend} then provides us with some values of
  $k\geq 1$, $\lambda_1,\dotsc,\lambda_k \in (0,1\rbrack$. Since
  $G$ is amenable, there exists a F\o lner sequence 
  $\{F_{j}\}_{j\geq 1}$ such that 
  $|F_j \setminus (F_j \cap s^{-1} F_j)| \leq (\epsilon/7) |F_j|$
  for all $s\in S$, $j\geq 1$,
  and such that, moreover, (\ref{eq:huhuhu}) holds with
  $\eta=\kappa/(24/\epsilon)^{k-1}$.
  Then Lemma \ref{lem:buend} gives us
  $\delta_0>0$ (called $\delta$ in the statement of
  Lemma \ref{lem:buend})
  and $N\geq 1$, as well as a finite set $S_0\subset G$
  (called $S$ in the statement of the Lemma) such that
  conclusions (\ref{conc:eins})--(\ref{conc:vier}) of Lemma \ref{lem:buend}
  hold for both $\phi_1$ and $\phi_2$ -- with
  respect to some collection
  $C_{1,1},\dotsc,C_{1,k}$ of subsets of $\{1,2,\dotsc,n\}$, in the case of
  $\phi_1$, and with respect to another collection $C_{2,1},\dotsc,C_{2,k}$ 
  of subsets of $\{1,2,\dotsc,n\}$, in the case of $\phi_2$.
  Set $S' = S_0 \cup S \cup S F_k \cup F_k^{-1}$ and
  $\delta = \min(\delta_0, 3 \epsilon^2/49 |F_k|)$.

By conclusion \eqref{conc:drei} of Lemma \ref{lem:buend},
the families $\{\phi_1(F_j) c\}_{1\leq j\leq k,\; c\in C_{i,j}}$,
  $i=1,2$, are $(\epsilon/7)$-disjoint.
We can therefore choose $F_{j,c}\subset F_j$ such that
$|F_{j,c}|\geq (1-2 \epsilon/7)|F_j|$ and both of the families
$\{\phi_1(F_{j,c}) c\}_{1\leq j\leq k,\; c\in C_{1,j}}$ and
$\{\phi_2(F_{j,c}) c \}_{1\leq j\leq k,\; c\in C_{2,j}}$ are families of disjoint
subsets. Moreover, by conclusion \eqref{conc:zwei}, for $i=1,2$,
$1\leq j\leq k$ and $c\in C_{i,j}$, the map $s\mapsto \phi_i(s) c$ from
$F_{j,c}$ to $\phi_i(F_{j,c})$ is injective.

We choose subsets $C_{i,j}'\subset C_{i,j}$, for $i=1,2$, $1\leq j\leq k$, such that
\[|C_{i,j}'| = \min_{i=1,2} |C_{i,j}|.\] Then, for $i=1,2$ and $1\leq j\leq k$,
\[\begin{aligned}&\left|\cup_{c\in C_{i,j}'} \phi_i(F_{j,c}) c\right| =
\sum_{c\in C_{i,j}'} |F_{j,c}| \geq
\left(1 - \frac{2 \epsilon}{7}\right) |F_j|\cdot
\min_{i=1,2} |C_{i,j}|\\
&\geq \left(1 - \frac{2 \epsilon}{7}\right) \min_{i=1,2} |\phi(F_j) C_{i,j}|
\geq \left(1 - \frac{2 \epsilon}{7}\right) \left(1 - \frac{\epsilon}{7}\right)
\lambda_j n \geq \left(1-\frac{3 \epsilon}{7}\right)
\lambda_j n,\end{aligned},\]
where we use conclusion \eqref{conc:vier} of \ref{lem:buend}.
Since $\lambda_1+ \dotsc + \lambda_k \geq 1 - \epsilon/7$, this implies,
by conclusion \eqref{conc:eins}, that, for $i=1,2$, the set
\[\Lambda_i = 
\bigcup_{1\leq j\leq k} \bigcup_{c\in C_{i,j}'} \phi_i(F_{j,c}) c\]
satisfies
\[\begin{aligned}|\Lambda_i| &=
\sum_{1\leq j\leq k} \left|\cup_{c\in C_{i,j}'} \phi_i(F_{j,c}) c\right|
= \sum_{1\leq j\leq k} \sum_{c\in C_{i,j}'} |F_{j,c}|\\
&\geq
\left(1-\frac{3 \epsilon}{7}\right) \left(\sum_{j=1}^k \lambda_j\right) n
\geq \left(1 - \frac{4 \epsilon}{7}\right) n.\end{aligned}\]
Since $|C_{1,j}'| = |C_{2,j}'|$ for all $1\leq j\leq k$,
we also see that $|\Lambda_1| = |\Lambda_2|$.

Choose a
bijection $\rho_j:C_{1,j}' \to C_{2,j}'$ for each $1\leq j\leq k$.
For every $x \in \Lambda_1$.
there are uniquely determined elements $1\leq j\leq k$,
$c\in C_{1,j}'$, $g\in F_{j,c}$, such that $x = \phi_1(g) c$.
We define $\tau(x) = \phi_2(g) \rho_j(c) \in \Lambda_2$.
We complete the definition of $\tau:\{1,\dotsc,n\}\to \{1,\dotsc,n\}$
by letting its restriction to $\{1,\dotsc,n\}\setminus \Lambda_1$
be an arbitrary bijection to $\{1,\dotsc,n\}\setminus \Lambda_2$.

Now let $s\in S$. For each $1\leq j\leq k$ and each $c\in C_{1,j}'$,
\[\begin{aligned}
|F_{j,c} \cap s^{-1} F_{j,c}| &\geq |F_{j,c}| - |F_j\setminus (F_j\cap
s^{-1} F_{j,c})| \\ &\geq 
|F_{j,c}| - (|F_j\setminus (F_j\cap s^{-1} F_j)| +
|s^{-1} F_j \setminus s^{-1} F_{j,c}|)\\
&\geq \left(1-\frac{2\epsilon}{7}\right)
|F_j| - \left(\frac{\epsilon}{7} |F_j| + \frac{2 \epsilon}{7} |F_j|\right)
= \left(1 - \frac{5 \epsilon}{7}\right) |F_j|,
\end{aligned}\]
where we use the assumption that $|F_j\setminus (F_j\cap s^{-1} F_j)| \leq
(\epsilon/7) |F_j|$. Hence
\[\begin{aligned}
\left|\cup_{c\in C_{1,j}'} \phi_1(F_{j,c}\cap s^{-1} F_{j,c}) c\right| &=
\sum_{c\in C_{1,j}'} |F_{j,c}\cap s^{-1} F_{j,c}|\geq
\left(1 - \frac{5 \epsilon}{7}\right) |F_j| \min_{i=1,2} |C_{i,j}|\\
&\geq \left(1 - \frac{5 \epsilon}{7}\right) \left(1 - \frac{\epsilon}{7}\right)
\lambda_j n = \left(1 - \frac{6 \epsilon}{7}\right)
\lambda_j n.\end{aligned}\]
Since $\lambda_1+\dotsc+\lambda_k \geq 1 - \epsilon/7$, this implies that
the set
\[\Lambda_{1,s} = \bigcup_{1\leq j\leq k} 
\bigcup_{c\in C_{1,j}'} \phi_i(F_{j,c}\cap s^{-1} F_{j,c}) c\]
has at least $(1-\epsilon+6 \epsilon^2/49) n$ elements.

Since $\phi_1$ and $\phi_2$ are $(S',\delta,n)$-sofic approximations,
we see that, for $i=1,2$, the set
\[R_i = \{x\in \{1,\dotsc,n\}: \phi(s) x = \phi(s g) \phi(g)^{-1} x\;\;\;\;\;\; \forall
g\in F_k\}\]
has at least $(1 - |F_k| \delta) n \geq (1 - 3 \epsilon^2/49) n$
elements. Thus, the set
\[\Lambda = \tau(\Lambda_{1,s})\cap \tau(R_1) \cap R_2\]
has at least $(1 - \epsilon + 6 \epsilon^2/49 - 2\cdot 3\epsilon^2/49) n
= (1-\epsilon) n$ elements.

Let $y\in \Lambda$. Then $\tau^{-1}(y) \in \Lambda_{1,s}$. In consequence,
there are uniquely determined elements $1\leq j\leq k$,
$c\in C_{1,j}'$, $g\in F_{j,c}\cap s^{-1} F_{j,c}$, such that
$\tau^{-1}(y) = \phi_1(g) c$; moreover, $y = \phi_2(g) \rho_j(c)$.
Since $\tau^{-1} y \in R_1$, we know that
\[(\phi_1(s))(\tau^{-1}(y)) = (\phi_1(s g) \phi_1(g)^{-1})(\phi_1(g) c) = \phi_1(s g) c.\]
Similarly, since $y\in R_2$, we know that
\[(\phi_2(s))(y) = \phi_2(s g) \rho_j(c).\]
Since $g\in s^{-1} F_{j,c}$, we know that $s g \in F_{j,c}$, and so
\[\tau(\phi_1(s g) c) = \phi_2(s g) \rho_j(c).\]
In other words,
\[(\tau \circ \phi_1 \circ \tau^{-1})(y) = (\phi_2(s))(y)\]
for all $y\in \Lambda$, as was desired.
\end{proof}

\begin{center}
  * * *
\end{center}

This is a good point at which to emphasize the relation with the work
of Arzhantseva and P\u{a}unescu \cite{MR3350728}, who introduced the
concept of {\em weakly stable groups}. Translated into the language used here,
their definition reads as follows: let $G$ be a finitely generated
group and $A$ a finite set of generators of $G$. The group $G$ is
said to be {\em weakly stable} if, for every $\epsilon>0$, there
are a $\delta>0$ and a finite subset $S'\subset G$ with $A\subset S'$
such that, for every $n\geq 1$
and every $(S',\delta,n)$-sofic approximation $\phi:S'\to \Sym(n)$
of $G$, there is a homomorphism $\phi':G\to \Sym(n)$ such that
$d_h(\phi(g),\phi'(g))<\epsilon$ for all $g\in A$.
(Actually, \cite{MR3350728} requires $S'$ to be the ball
$B(r) = \{g_1 \dotsc g_k : g_i\in A\cup A^{-1}, k\leq r\}$
for $r = 1/\delta$,
but it is easy to see that the definition thus obtained is
equivalent to the one given here.)

Theorem 1.1 of \cite{MR3350728} states that a finitely generated amenable
group $G$ is weakly stable if and only if it is residually finite.
Let us see how to prove this using Prop.~\ref{prop:labil}.
(We make no claim to originality here; we are simply showing
how to do matters in elementary language, without using ultraproducts.
In particular,
the procedure in \cite[\S 6]{MR3350728} is very close to what we will do.)
Let $G$ be finitely generated and amenable.
It is easy to show that weak stability implies residual finiteness. Let us prove the converse. Assume $G$ is residually finite.
Let $\epsilon>0$. Proposition \ref{prop:labil}
(with $S=A$) gives us a finite subset $S'\subset G$ with $A\subset S'$
and constants $N\in \mathbb{Z}^+$, $\delta>0$. 
Since $G$ is
residually finite, there exists a homomorphism $\phi_0:G\to H$,
$H$ finite,
such that $\phi_0(g)\ne e$ for every $g\in S'$ with $g\ne e$. We
compose $\phi_0$ with the map $\rho:H\to \Sym(H) \sim \Sym(n_0)$, $n_0=|H|$,
defined by the action of $H$ on itself by left multiplication, and obtain
a homomorphism $\rho\circ \phi_0 :G\to \Sym(n_0)$ such that, for every
$g\in S'$ with $g\ne e$, $(\rho\circ \phi_0)(g)$ has no fixed points.
Set $\delta' = \min(\epsilon,\delta/n_0,1/N)$.
Now let an $(S',\delta',n)$-sofic approximation
$\phi:S'\to \Sym(n)$ be given. If $n< \max(N,n_0/\delta)\leq 1/\delta'$, then
$\phi$ is actually a homomorphism, and we are done. Assume
that $n\geq \max(N,n_0/\delta)$. Let $\phi_1:S'\to \Sym(n)$
be the composition of the direct product
$(\rho\circ \phi_0)^\ell:S'\to \Sym(\ell n_0)$
($\ell$ copies of $\rho\circ \phi_0$, $\ell = \lfloor n/n_0\rfloor$)
with the embedding $\Sym(\ell n_0) \to \Sym(n)$ induced by the
inclusion $\{1,2,\dotsc,\ell n_0\} \to \{1,2,\dotsc,n\}$.
Then $\phi_1$ is a homomorphism such that $\phi_1(g)$ has
$< n_0 \leq \delta n$ fixed points for every $g\in S'$, $g\ne e$;
in particular, it is an $(S',\delta,n)$-sofic approximation.
Since $\phi$ is also an $(S',\delta,n)$-sofic approximation,
we may apply Prop.~\ref{prop:labil}, and obtain a bijection
$\tau:\{1,2,\dotsc,n\}\to \{1,2,\dotsc,n\}$ such that, for
$\phi':G\to \Sym(n)$ defined by $\phi'(g) = \tau \circ \phi_1(g) \circ
\tau^{-1}$ for $g\in G$,
\[d_h(\phi(g),\phi'(g)) < \epsilon\]
for all $g\in A$, $g\ne e$, and
$d_h(\phi(e), \phi'(e)) =
d_h(\phi(e),\id) < \delta' \leq \epsilon$.
Since $\phi'$ is a homomorphism, we have proved that $G$ is weakly stable.


\section{Baumslag-Solitar groups and the Higman group}\label{sec:lindur}

The {\em Baumslag-Solitar group} $\BS(1,m)$ is defined by
 \begin{equation}\label{eq:harmo}
\BS(1,m)=\langle a_1,a_2: a_1^{-1} a_2 a_1=a_2^m\rangle,\end{equation}
where $m\geq 2$.
It is clear why we care about $\BS(1,m)$: the elements $a_1$, $a_2$ of the Higman group
generate a group isomorphic to $\BS(1,2)$. 

 Fix $M\geq 1$. Then
 the sets 
\begin{equation}\label{eq:prico}
F_n= \{a_1^i a_2^j:\; 0\leq i< 2 n,\; 0\leq j< 2 M m^{2 n}\},\;\;\;
\;\;\;\; n\geq 1,\end{equation}
 satisfy $$|s F_n\Delta F_n|\leq |F_n|/n \text{ for all } s \in \{a_1,a_2\},$$
 and so $\BS(1,m)$ is amenable.

\begin{proof}[Proof of Theorem \ref{thm:jut}]
  Assume the group $H_{4,m}$ is sofic. We know that, for every sofic group
  $H$, every semidirect product of the form $F\ltimes H$, $F$ finite, is
  sofic. (This is true more generally for $F$ amenable \cite{MR2220572}.)
  Hence, the following semidirect product is sofic:
  \[G = (\mathbb{Z}/4 \mathbb{Z}) \ltimes H_{4,m},\]
  where we define the action of $\mathbb{Z}/4 \mathbb{Z}$ as follows,
  in terms of a generator $t$ of $\mathbb{Z}/4 \mathbb{Z}$:
  \[\begin{aligned}
  t a_1 t^{-1} = a_2,\;\;\;\;\;\;
  t a_2 t^{-1} = a_3,\;\;\;\;\;\;
  t a_3 t^{-1} = a_4,\;\;\;\;\;\;
  t a_4 t^{-1} = a_1.\end{aligned}\]

 Set $S'\subset G$ finite, $\delta>0$, we will specify them later.
  Since $G$ is sofic, there
  is an $(S',\delta/2,k)$-sofic approximation $\alpha$ of $G$ for some 
$k\geq 1$.
  Then, for any $r\geq 1$, the direct product of $r$ copies of $\alpha$
  is an $(S',\delta/2,r k)$-sofic approximation of $G$.
Assume from now on that $n\geq 2 k/\delta$. Let
$r = \lfloor n/k\rfloor$. Then
\begin{equation}\label{eq:hudd}
  n\geq r k > n - k \geq \left(1 - \frac{\delta}{2}\right) n.
  \end{equation}
  We embed $\Sym(r k)$ in $\Sym(n)$, and obtain
  an $(S',\delta,n)$-sofic approximation $\phi$ of $G$.
  
  Clearly, $\phi$ restricts to an $(S_0',\delta,n)$-sofic approximation
  $\phi|_{\BS(1,2)}$ of $\BS(1,m) = \langle a_1,a_2 | a_1^{-1} a_2 a_1 = a_2^m
  \rangle < G$, where $S_0' = S'\cap \BS(1,m)$.
  Since $\BS(1,m)$ is amenable,
  we will be able to use Prop.~\ref{prop:labil} (``any two sofic
  approximations of
  an amenable group are conjugate''). Now, $\BS(1,m)$ has an
  $(S_0',\delta,n)$-sofic
  approximation $\psi$ that is easily described: identifying
  the set $\{1,\dotsc,n\}$ with $\mathbb{Z}/n \mathbb{Z}$, we define
  \[\begin{aligned}
  \psi(a_2) &= (x\mapsto x-1 \mo n)\\
  \psi(a_1) &= (x\mapsto m^{-1} x \mo n).\end{aligned}\]
  (We recall $n$ is coprime to $m$.) 
  This defines a homomorphism from $\BS(1,m)$ to $\mathbb{Z}/n \mathbb{Z}$,
  i.e., condition~(\ref{it:amarg}) in Definition \ref{def:sofic} holds for
  $\delta$, $S_0'$ completely arbitrary (that is, \[d_h(\phi(g) \phi(h),\phi(gh))=0\]
  for all $g,h\in \BS(1,m)$). It remains to check rule (\ref{it:omorg})
  in Def.~\ref{def:sofic}; let us do this.

  Given a non-trivial reduced word
  $w = a_{i_1}^{r_1} a_{i_2}^{r_2} \dotsb a_{i_k}^{r_k}$, $i_j=1,2$, the map
  $\psi(w)$ is a linear polynomial map $x\mapsto P(x) \mo n$ from 
  $\mathbb{Z}/n \mathbb{Z}$ to itself, where $P(x) = m^a x + b$,
  $a,b\in \mathbb{Z}$, $a$ and $b$ depending only on $w$ and
  not both equal to $0$. An element $x\in \mathbb{Z}/n\mathbb{Z}$ is a
  fixed point of $x\mapsto P(x) \mo n$ if and only if $(m^a-1) x = - b
  \mo n$.
  Thus, $\psi(w)$ has at most $m^a-1 = O_w(1)$ fixed points, unless
  $m^a\equiv 1 \mo n$ and $b\equiv 0 \mo n$. Clearly, if $a\ne 0$
  and $m^a\equiv 1 \mo n$, then $m^{|a|}\geq n+1$; if $b\ne 0$
  and $b\equiv 0\mo n$, then $|b|\geq n$. At the same time,
  $|a|\leq \ell$ and $|b|\leq m^\ell$, where $\ell = \sum_j |r_j|$.
It is also clear that $a=0$, $b=0$ only when $w$ equals the identity in
$\BS(1,m)$.
 Hence,
  for $S'$ given, $\psi$ is an $(S_0',\delta,n)$-sofic approximation
  provided that $n$ is larger than a constant depending only on
  $S_0'$ and $\delta$, something that we can assume.

  We can hence apply Proposition \ref{prop:labil}, and obtain a bijection
  $\tau:\{1,2,\dotsc,n\} \to \mathbb{Z}/n \mathbb{Z}$ such that
  \[(\tau \circ \phi(s) \circ \tau^{-1})(x) = (\psi(s))(x)\]
  for all $s\in S_0$ and $\geq (1-\epsilon) n$ values of
  $x\in \mathbb{Z}/n \mathbb{Z}$, where $S_0\subset \BS(1,m)$ and
  $\epsilon>0$ are arbitrary (and $S_0'$ and $\delta>0$
  are set in terms of $S_0$ and $\epsilon$).

  What remains is routine. Let $S_0 = \{a_1, a_2\}$. Let 
  \[S' = S'_0 \cup \{e, a_1, a_2, a_2^{-1}, t a_1, t, t^{-1}, t^{-2}, t^{-3}\}.\]
  (That is, we may add those elements to $S'$ if needed.)
 Since $t a_1 t^{-1} a_2^{-1} = e$
  in $G$, we see that
  \[(\phi(t) \phi(a_1) \phi(t^{-1}))(x) = \phi(a_2)(x)\]
  for $\geq (1-O(\epsilon)) n$ values of $x\in \mathbb{Z}/n\mathbb{Z}$,
  and so
  \[\begin{aligned}
  (\tau \phi(t^{-1}) \tau^{-1})^{-1}
  &\psi(a_1) (\tau \phi(t^{-1}) \tau^{-1}) (x) \\ &=
  (\tau \phi(t^{-1})^{-1} \tau^{-1})
  (\tau \phi(a_1) \tau^{-1}) (\tau \phi(t^{-1}) \tau^{-1}) (x) \\&=
  (\tau (\phi(t) \phi(a_1) \phi(t)^{-1}) \tau^{-1})(x) =
  (\tau \phi(a_2) \tau^{-1})(x) = (\psi(a_2))(x)\end{aligned}\]
  for $\geq (1-O(\epsilon)) n$ values of $x\in \mathbb{Z}/n\mathbb{Z}$.
  Defining $g:\mathbb{Z}/n\mathbb{Z}\to \mathbb{Z}/n \mathbb{Z}$
  by $g = \tau \phi(t^{-1}) \tau^{-1}$, we see that this means that
  \[g^{-1}(m^{-1} g(x)) = x-1\]
  for $\geq (1-O(\epsilon)) n$ values of $x$, and so
  \[g(y+1) = m g(y)\]
  for $\geq (1-O(\epsilon)) n$ values of $y = x-1 \in
  \mathbb{Z}/n\mathbb{Z}$.

  At the same time, because $t^{-4} = e$ in $G$, we know that
  \[x = \phi(e) x = \phi(t^{-4})(x) = (\phi(t^{-1}))^4(x)\]
  for $\geq (1-O(\epsilon)) n$ values of $x\in \mathbb{Z}/n \mathbb{Z}$,
  and so
  \[g^4(y) = (\tau \phi(t^{-1}) \tau^{-1})^4(y) = y\]
  for $\geq (1-O(\epsilon)) n$ values of $y\in \mathbb{Z}/n \mathbb{Z}$.
  Let $V$ be the set of all such values of $y$; clearly,
  $V' = V \cap g^{-1} V \cap g^{-2} V \cap g^{-3} V$ has size $|V'|\geq
  (1-O(\epsilon)) n$.

Let $f(y) = g(y)$ for $y\in V'$ and $f(y) = y$ for $y\notin V'$.
  Then
  \[f^4(y) = y\]
  for all $y\in \mathbb{Z}/n\mathbb{Z}$, and
  \[f(y+1) =  m f (y)\]
  for $\geq (1-O(\epsilon)) n$ values of $y\in \mathbb{Z}/n \mathbb{Z}$,
  as desired.
\end{proof}

As we mentioned in the introduction, and as should be clear from above,
the same proof works if one or more relations are added to the relations
\begin{equation}\label{eq:jo1}\begin{aligned}
t^4 = e,\;\;\;\;\;
  t a_1 t^{-1} = a_2,\;\;
  t a_2 t^{-1} = a_3,\;\;
  t a_3 t^{-1} = a_4,\;\;
  t a_4 t^{-1} = a_1,\;\; a_1^{-1} a_2 a_1 = a_2^m
\end{aligned}\end{equation}
defining
$(\mathbb{Z}/4\mathbb{Z}) \ltimes H_{4,m}$. 
Assume that $\BS(1,m)$ embeds (by means of the map taking $a_1$ to $a_1$
and $a_2$ to $a_2$) in the quotient of 
$(\mathbb{Z}/4\mathbb{Z}) \ltimes H_{4,m}$ obtained in this way.
(Just to give an example --
 a quick check via GAP suggests that this is the case when the relation
being added is $(a_1 a_3)^3 = e$.)
We can then go through the proof above, and obtain a result similar to
Theorem \ref{thm:jut}.


For instance: assuming that $\BS(1,m)$ does embed in the group $G$
given by the relation $(a_1 a_3)^3 = e$ and the relations in (\ref{eq:jo1}), 
we obtain that, if $G$ is sofic,
then, for
 any $\epsilon>0$, there is an $N>0$ such that, for every $n\geq N$, there
is a bijection $f:\mathbb{Z}/n\mathbb{Z}\to
\mathbb{Z}/n\mathbb{Z}$ for which, for $g(x) = f^2(f^{-2}(x)+1)$,
\[f(x+1) = m f(x),\;\;\;\;\;\; g(g(g(x)+1)+1)+1 = x\]
for at least $(1-\epsilon) n$ elements $x$ of
$\mathbb{Z}/n\mathbb{Z}$, and, moreover,
\[f(f(f(f(x))))=x\]
for all $x\in \mathbb{Z}/n\mathbb{Z}$. Notice that $g(x)$ behaves
$(m\cdot)$-locally like a constant times $x^m$, meaning that, for at least
$(1-O(\epsilon)) n$ elements $x$ of $\mathbb{Z}/n\mathbb{Z}$,
\[g(m x) = m^m g(x).\]
This can be easily seen as follows: if $f(y+1)=m y$ is valid
for $y=f^{-1}(x)$, $y=f^{-1}(f^{-1}(x)+1)$, $y = f^{-2}(x)$ and
$y = m f^{-1}+k$, $0\leq k \leq m-1$, then
\begin{equation}\label{eq:praktisch}\begin{aligned}
g(m x) &= f^2(f^{-2}(m x) + 1) = f^2(f^{-1}(f^{-1}(x)+1)+1) =
f(m (f^{-1}(x)+1))\\
&= f(m f^{-1}(x) + m) = m f(m f^{-1}(x) + m - 1) = \dotsc = m^m f(m f^{-1}(x))
\\ &= m^m f(f(f^{-2}(x)+1)) = m^m g(x).
\end{aligned}\end{equation}
Of course, we obtain the same conclusion, without the condition $f(f(f(f(x))))=x$, if we remove the relation $t^4=1$; we did not use the condition
$f(f(f(f(x))))=x$ in any step of (\ref{eq:praktisch}).

\section{Few cycles of length $3$}\label{sec:fecy}

The following statement resembles Theorem \ref{thm:mathov}, but is neither
weaker nor stronger. It will follow readily from the fact that
the group $H_3$ is trivial. It should be clear that the proof would
work just as well for analogous statements corresponding to any other
finite presentation of the trivial group.

\begin{lemma}\label{lem:higo3}
  There is a $\delta>0$ such that, for every coprime integers $n,m>1$ 
  and every bijection $f:\mathbb{Z}/n\mathbb{Z}\to \mathbb{Z}/n\mathbb{Z}$,
  there are at least $\geq \delta n$ values of $x\in \mathbb{Z}/n\mathbb{Z}$
  for which at least one of the equations
  \begin{equation}\label{eq:mezo}
    f(x+1)=m f(x),\;\;\;\;\;\;\;\;\;\;
    f(f(f(x)))=x\end{equation}
  fails to hold.
\end{lemma}
\begin{proof}
Assume the statement of the Lemma is false; that is, assume
that the equations (\ref{eq:mezo}) hold
for $\geq (1-\delta) n$
values of $x\in \mathbb{Z}/n\mathbb{Z}$, where $\delta>0$ is arbitrarily small. 
We know that the group
 $H_3$ generated by elements $a_1$, $a_2$, $a_3$ satisfying
the relations
\[\begin{aligned}
a_1^{-1} a_2 a_1 &= a_2^m,\\
a_2^{-1} a_3 a_2 &= a_3^m,\\
a_3^{-1} a_1 a_3 &= a_1^m
\end{aligned}\]
is trivial \cite{MR0038348}. In other words,
the normal closure of the subgroup of $\mathbb{F}_3$ generated by the
words 
\[\begin{aligned}
w_1(a_1,a_2,a_3) &= a_1^{-1} a_2 a_1 a_2^{-m},\\
w_2(a_1,a_2,a_3) &= a_2^{-1} a_3 a_2 a_3^{-m},\\
w_3(a_1,a_2,a_3) &= a_3^{-1} a_1 a_3 a_1^{-m}\end{aligned}\]
equals the free group $\mathbb{F}_3$ itself.
In particular, $a_1$ is equal to a product of
conjugates of $w_1$, $w_2$, $w_3$ and their inverses.

Let $g_1,g_2,g_3:\mathbb{Z}/p\mathbb{Z} \to \mathbb{Z}/p\mathbb{Z}$ be
defined by
\[g_1(x) = x-1,\;\;\;\;\;\;\;g_3(x) = f(f^{-1}(x)-1),
\;\;\;\;\;\;\;g_2(x) = f^2(f^{-2}(x)-1).\]
It is easy to check that, if (\ref{eq:mezo}) holds for $\geq (1-\delta) n$
values of $x$, then
\[\begin{aligned}
(g_1 g_3)(x) &= f(f^{-1}(x)-1)-1 = \frac{1}{m} f(f^{-1}(x)) - 1 =
\frac{x}{m} - 1 \\ &= \frac{x-m}{m} = f(f^{-1}(x-m) - 1) = (g_3 g_1^m)(x)
\end{aligned}\]
holds for $\geq (1 - 2\delta) n$ values of $x$, and so
\[d_h(w_3(g_1,g_2,g_3),e) \leq 2\delta.\]
Since $g_3 = f g_1 f^{-1}$, $g_2 = f g_3 f^{-1} = f^2 g_1 f^{-2}$ and
$g_1 = f^3 g_1 f^{-3} = f^2 g_3 f^{-2}$, we deduce that we also have
\[d_h(w_i(g_1,g_2,g_3),e) = O(\delta)\]
for $i=1,2$.
Hence,
\[d_h(w\cdot w_i(g_1,g_2,g_3)\cdot w^{-1},e) = O_w(\delta)\]
for any word $w\in \mathbb{F}_3$ and any $i=1,2,3$.

As we were saying,
$a_1$ is equal to a product of conjugates of $w_1$, $w_2$, $w_3$
and their inverses; therefore,
\[d_h(g_1,e) = O(\delta).\]
At the same time, it is trivial that $d_h(g_1,e)=1$. Setting $\delta$
smaller than a constant, we obtain a contradiction.
\end{proof}

Theorem \ref{thm:mathov} is a different matter. On the one hand,
the map $f_{m,n}$ in the statement of Theorem \ref{thm:mathov}
is actually an exponentiation map, instead of merely behaving locally
like one. On the other hand, the conclusion of Theorem \ref{thm:mathov}
is stronger than that of Lemma \ref{lem:higo3}: we will see that
the proportion of elements of $\mathbb{Z}/p\mathbb{Z}$ fixed by $f_{m,n}$
actually goes to $0$, as opposed to just being bounded away from $1$.

Let us first give a sketch of the proof of Thm.~\ref{thm:mathov}.
Suppose \begin{equation}\label{eq:wittlorc}m^{m^{m^x}} = x\end{equation}
 for a positive proportion of all $x$. Then there is
a bounded $k$ such that (\ref{eq:wittlorc}) holds for both $x$ and $x+k$
for a positive proportion of all $x$. (This is a simple fact that
we will derive explicitly; some readers will recognize the idea of  
Poincar\'e recurrence at work here.) For a positive proportion of
all $x\in \{0,1,\dotsc,p-1\}$, we should thus have
\[m^{m^{m^{x+k}}} = x + k = m^{m^{m^x}}+k;\]
writing $y = m^{m^x}$, we obtain that
\begin{equation}\label{eq:lakapo}m^{y^{m^k}} = m^y+k\end{equation}
for a positive proportion of all $y$. By the same argument as before,
there is a bounded $\ell$ such that (\ref{eq:lakapo}) is true for both
$y$ and $\ell y$ for a positive proportion of all $y$. Writing 
$z = m^y$, we seem to obtain that
\begin{equation}\label{eq:azko}(z+k)^{\ell^{m^k}} = z^l+k\end{equation}
for a positive proportion of all $z$. However, (\ref{eq:azko}) is a
non-trivial polynomial equation, and thus has a finite number of roots,
given us a contradiction.

We need to be a little more careful with this. For one thing, 
$f_{m,n}(x)$ is not exactly an exponentiation map; the exponentiation
map $x\to m^x \mo n$ 
is defined from $\mathbb{Z}/ord_n(m) \mathbb{Z}$ to $\mathbb{Z}/n\mathbb{Z}$,
not from $\mathbb{Z}/n\mathbb{Z}$ to itself.
Let us write out a correct proof in detail.

\begin{proof}[Proof of Thm.~\ref{thm:mathov}]
Let $f_{m,n}$ be as in the statement. 
Suppose $f_{m,n}(f_{m,n}(f_{m,n}(x)))=x$ for all $x$
in a subset $X$ of $\{0,1,\dotsc,n-1\}$ of size $|X|\geq \delta n$,
where $\delta>0$; we will show that this leads to a contradiction for
$\delta$ sufficiently small and $n$ larger than a constant depending on
$m$ and $\delta$.
We can assume that the order of $m$
in $(\mathbb{Z}/n\mathbb{Z})^*$ is at least $\delta n$, as otherwise
the image of $f_{m,n}$ would be contained in a set of size $< \delta n$, and
we would have a contradiction immediately.

If there were more than $\delta n/2$ elements $x$ such that the element
$x'$ of $X$ immediately after $x$ were at distance at least $2/\delta$
from $x$ in the natural
cyclic ordering ($0,1,2,\dotsc$) for $\mathbb{Z}/n\mathbb{Z}$, we would get a contradiction: the total distance traversed by going 
through the elements $x_1, x_2,\dotsc,x_m$ of $S$ and then from $x_m$ to $x_1$
would be more than $n$. Hence there is a constant
$1\leq k<2/\delta$ such that, for at least 
$(\delta n/2)/(2/\delta) - 1 = \delta^2 n/4 - 1$ elements 
$x$ of $X$, 
$x+k$ is also in $X$. (The ``$-1$'' term is here simply because
we want $x+k$, and not just $x+k$ reduced modulo $n$, to be in $X$.)
We can assume that $n\geq \sqrt{20/\delta}$, so
that $\delta^2 n/4 - 1\geq \delta^2 n/5$. We thus have
$|X\cap (X-k)|\geq \delta^2 n/5$.

For all $x\in X\cap (X-k)$,
\[f_{m,n}(f_{m,n}(f_{m,n}(x)))+k =
x + k = f_{m,n}(f_{m,n}(f_{m,n}(x+k))) =
f_{m,n}(f_{m,n}(\overline{m^k f_{m,n}(x)})),\]
where, given $a\in \mathbb{Z}$, we write $\overline{a}$ for the element of
$\{0,1,\dotsc,n-1\}$ such that $\overline{a}\equiv a \mod n$. Obviously,
$\overline{m^k f_{m,n}(x)} = m^k f_{m,n}(x) - c n$ for some $0\leq c< m^k$. Writing
$y$ for $f_{m,n}(f_{m,n}(x))$, and noting that $f_{m,n}(x) = \overline{m^x}$, we obtain that
\begin{equation}\label{eq:pale}
f_{m,n}(y)+k = f_{m,n}\left(\overline{m^{m^k f_{m,n}(x) - c n}}\right)
= f_{m,n}\left(\overline{c' \left(m^{f_{m,n}(x)}\right)^{m^k}}\right) = 
f_{m,n}\left(\overline{c' y^{m^k}}\right),
\end{equation}
where $c'\in \{1,\dotsc,n-1\}$ is such that $c' \equiv m^{- c n} \mod n$.
Because the order of $m$ is at least $\delta n$, the map $x\mapsto m^x \mo n$ can send at most
$1/\delta$ elements to the same element; hence, there are at least
 $\delta^2 \cdot
\delta^2 n/5 = \delta^4 n/5$ values of $y$ for which (\ref{eq:pale}) holds.
Since there are $m^k$ possible values of $c$, there are at most $m^k$
possible values of $c'$; choose one such that 
\begin{equation}\label{eq:kurko}
f_{m,n}(y) +k = f_{m,n}\left(\overline{c' y^{m^k}}\right)
\end{equation}
holds for at least $\delta^4 n/5 m^k$ elements $y\in \{1,\dotsc,n-1\}$.

Now, by the same argument
as before, this implies that
there is a constant $\ell = m^r$, $1\leq r<10 m^k/\delta^4$,
such that there are at least $\delta^8 n/ 100 m^k$ elements  
$y\in \{1,\dotsc,n-1\}$ for which
\begin{equation}\label{eq:qian}
f_{m,n}(y) +k = f_{m,n}\left(\overline{c' y^{m^k}}\right)\;\;\;\;\;\text{and}
\;\;\;\;\; f_{m,n}(\ell y) +k = f_{m,n}\left(\overline{c' (\ell y)^{m^k}}\right).
\end{equation}
(Note no ``$-1$'' term is needed now.)

Write $z$ for $f_{m,n}(y)$. Then $f_{m,n}(\ell y) = \overline{m^{\ell y}}\equiv z^\ell
 \mod n$, and, similarly,
\[f_{m,n}\left(\overline{c' (\ell y)^{m^k}}\right) = 
\overline{m^{\overline{c' (\ell y)^{m^k}}}} = \overline{
m^{\ell^{m^k} \overline{c' y^{m^k}} - \kappa n}} = \overline{\kappa'  
\left(m^{\overline{c' y^{m^k}}}\right)^{\ell^{m^k}}} = 
\overline{\kappa' f_{m,n}\left(\overline{c' y^{m^k}}\right)^{\ell^{m^k}}}
\]
where $0\leq \kappa < \ell^{m^k}$
and $\kappa'\in \{1,\dotsc,n-1\}$ is such that 
$\kappa' = m^{-\kappa n} \mo n$. Hence, we obtain from (\ref{eq:qian}) that
\[z^\ell + k \equiv \kappa' (z+k)^{\ell^{m^k}} \mo n,\]
or, what is the same,
\begin{equation}\label{eq:rictu}
  (z+k)^{\ell^{m^k}} - m^{\kappa n} (z^\ell + k) \equiv 0 \mo n.\end{equation}
This is, of course, an equation of the form $P(z) \equiv 0 \mo n$, where
$P$ is a non-constant polynomial with integer coefficients.

There are two possible
ways to proceed here. One would be to show that the discriminant
of this polynomial is non-zero, and then bound its common factor with $n$.
We follow an alternative route.\footnote{We thank I. Shparlinski for this
suggestion.} By a result of Konyagin's \cite{MR542556}, \cite{MR548524}, 
the number of roots $\mo n$ of a non-zero polynomial of degree $d$ is at most
$c_d n^{1-1/d}$, where $c_d = d/e + O(\log^2 d)$. Our polynomial
$P(z)$ is certainly non-trivial (it has leading coefficient $1$); hence,
the number of roots of $P(z)\equiv 0\mo n$ is
\[O_{\ell,m,k}\left(n^{1-1/\ell^{m^k}}\right).\]
Of course, there are different possible choices of $k$ and $\ell$,
but, since there are at most $2/\delta$ and $10 m^k/\delta^4$ such choices,
respectively,
and since both $k$ and $\ell$ are bounded in terms of $m$ and $\delta$
(really just in terms of $\delta$, in the case of $k$),
the total number of values of $z$ that are roots of
(\ref{eq:rictu}) for some possible $k$, $\ell$ is
\[O_{m,\delta}\left(n^{1-\eta_{m,\delta}}\right),\]
where $\eta_{m,\delta}>0$ depends only on $m$ and $\delta$.

At the same time, the number of elements $z = f_{m,n}(y)$ we are considering
is at least \[\delta\cdot \frac{\delta^8 n}{100 m^k} =
\frac{\delta^9 n}{100 m^k} \geq \frac{\delta^9}{100 m^{2/\delta}} n.\]
Thus, we obtain a contradiction provided that $n$ is larger than
a constant depending only on $m$ and $\delta$.
\end{proof}


\section{Fixed points}\label{sec:padic}

Following the example of \cite{MR2999154},
we will use ideas from $p$-adic analysis to prove Proposition \ref{prop:norvi}.

\begin{proof}[Proof of Prop.~\ref{prop:norvi}]
  Let  $f:\mathbb{Z}/p^r \mathbb{Z} \mapsto \mathbb{Z}/p^r
  \mathbb{Z}$, $r\geq 5$, be given. Assume that 
\[f(f(f(f(x))))=x\]
for at least $p^r/2$ values of $x$, and
  \[f(x+1) = m f(x),\;\;\; f(f(f(f(x))))=x\]
  for at least $p^r - c\cdot p^{r/4}$ values of $x\in \mathbb{Z}/p^r \mathbb{Z}$,
  where $c>0$ will be set later.
  Define  $g:\mathbb{Z}/p^r \mathbb{Z} \mapsto \mathbb{Z}/p^r \mathbb{Z}$
  by $g(x) = m^{-1} f(m x)$. Then
  \[g(x+1) = m^{p-1} g(x)\]
  for at least $p^r - (p-1) c\cdot p^{r/4}$ values of 
$x\in \mathbb{Z}/p^r \mathbb{Z}$.

  Hence, there are $c_1,\dotsc,c_k \in
  \mathbb{Z}/p^r \mathbb{Z}$, $c_j\ne 0$, $k\leq (p-1) c \cdot p^{r/4}$,
  such that,
  for every $x\in \mathbb{Z}/p^r \mathbb{Z}$, there is a $1\leq j\leq k$
  such that $g(x) = g_j(x)$, where $g_j:\mathbb{Z}/p^r \mathbb{Z}\to
  \mathbb{Z}/p^r \mathbb{Z}$ is defined by
  \[g_j(x) = c_j \cdot s^x,\]
where $s = m^{p-1}$.

A remark on the definition of the maps $g_j$ is in order.
By Fermat's little theorem, $s\equiv 1 \mo p$. Since the kernel of
the reduction $\mo p$ map $(\mathbb{Z}/p^r\mathbb{Z})^*\to \mathbb{Z}/p
\mathbb{Z}$ has order $p^{r-1}$, it follows that the map
$x\to s^x$ has period dividing $p^{r-1}$, and thus 
  is well-defined as a map from $\mathbb{Z}/p^r \mathbb{Z}$ to
  itself. (It is, in fact, well-defined as a function from the
  $p$-adic ring $\mathbb{Z}_p$ to itself.) The maps $g_j$ are thus
well-defined.

Thus, if $g(g(g(g(x))))=x$, there
  must be $1\leq j_1, j_2, j_3,j_4\leq k$ such that
  \[g_{j_4}(g_{j_3}(g_{j_2}(g_{j_1}(x)))) = x.\]
  This implies immediately that $(x,g_{j_1}(x),g_{j_2}(x),g_{j_3}(x))$
  is a fixed point of the function
  $G$ from $(\mathbb{Z}/p^r \mathbb{Z})^4$ to itself given by
  \[\begin{aligned}
  G(x_1,x_2,x_3,x_4) &= (g_{j_4}(x_4), g_{j_1}(x_1), g_{j_2}(x_2),
  g_{j_3}(x_3))\\ &= \left(c_{j_4} \cdot s^{x_4}, c_{j_1} \cdot s^{x_1},
  c_{j_2} \cdot s^{x_2}, c_{j_3} \cdot s^{x_3}\right).\end{aligned}\]
  
  Our aim is now to show that $G$ can have at most one fixed point.
  Since $G$ is actually well-defined on the $p$-adics, we could do
  this by an appeal to Hensel's lemma and a brief argument
  involving the $p$-adic metric,
  but we can do without that (even though that is clearly
  the idea in what follows).

  Let $(x_1,x_2,x_3,x_4)\in (\mathbb{Z}/p^r \mathbb{Z})^4$
  be a fixed point of $G$. Since $s^x\equiv 1 \mo p$ for every
  $x\in \mathbb{Z}/p^r \mathbb{Z}$, this implies that
  $(x_1,x_2,x_3,x_4)\equiv (c_{j_4},c_{j_1},c_{j_2},c_{j_3}) \mo 3$.
  Now, the congruence class $x\mo p$ determines $s^x \mo p^2$; thus,
  we know $G(x_1,x_2, x_3,x_4) = (x_1,x_2,x_3,x_4) \mo p^2$. In general,
  the congruence class $x\mo p^k$ determines \[s^x \mo p^{k+1},\]
  and so, iterating, we find that we know $G(x_1,x_2,x_3,x_4) = (x_1,x_2,x_3,x_4)$ modulo $p^3$, $p^4$,\dots, and, lastly, $\mo p^r$. In other words,
  $G$ has at most one fixed point.

  Hence, the number of fixed points of all such $G$ is at most 
$k^4 \leq ((p-1) c)^4\cdot p^r$. Hence, by assumption, $((p-1) c)^4 \leq 1/2$.
We get a contradiction for $c = 1/2^{1/4} p$.
\end{proof}

The same kind of argument can be applied to other relators. For instance,
consider the function $g$ discussed at the end of \S \ref{sec:lindur}:
a bijection $g:\mathbb{Z}/n\mathbb{Z} \to \mathbb{Z}/n\mathbb{Z}$
such that 
$g(m x) = m^m g(x)$ for 
$(1-\epsilon_n) n$ values of $x\in \mathbb{Z}/n\mathbb{Z}$
and $g(g(g(x)+1)+1)+1 = x$ for either $(1-\epsilon_n) n$ or even just
(say) $n/8$ values of  $x\in \mathbb{Z}/n\mathbb{Z}$.
It is easy to see how this leads to a contradiction for $\epsilon_n$
less than a constant $c$ times $n^{-2/3}$. Let us do this for $n$ equal to
a prime $p$ not dividing $m$, since
$\mathbb{Z}/p\mathbb{Z}$ gives us a nicer framework
than $\mathbb{Z}/p^r \mathbb{Z}$ for this sort of function $g$.

There are $c_1,\dotsc,c_k\in \mathbb{Z}/p\mathbb{Z}$, $c_j\ne 0$,
$k\leq \epsilon_p p < c p^{1/3}$, such that, for every
$x\in \mathbb{Z}/p\mathbb{Z}$, there is a $1\leq j\leq k$ such that
$g(x) = g_j(x)$, where $g_j:\mathbb{Z}/p\mathbb{Z} \to \mathbb{Z}/p\mathbb{Z}$
is defined by
\[g_j(x) = c_j x^m.\]
Thus, if $g(g(g(x)+1)+1)+1 = x$, there must be $1\leq j_1,j_2,j_3\leq k$ such
that
\[(c_{j_3} (c_{j_2} (c_{j_1} x^m + 1)^m + 1)^m) + 1 = x.\]
There are at most $m^3$ solutions to this equation for $j_1$, $j_2$, $j_3$
given. Hence, the total number of solutions is at most $m^3 k^3 < (m c)^3 p$.
We obtain a contradiction for $c<1/2 m$, since then
$(m c)^3 p < p/8$.

\section{Heuristics}\label{sec:heur}

Let us now address the question: is the existence of a function $f$
as in Thm. \ref{thm:jut} plausible?

Let $f$ be an element taken uniformly at random from the group 
$\Sym(n) = \Sym(\mathbb{Z}/n\mathbb{Z})$ of all bijections from
$\mathbb{Z}/n\mathbb{Z}$ to itself. Let $P_n$ be the probability that
$f\circ f\circ f\circ f = e$. It is clear that, for $n\geq 5$,
\begin{equation}\label{eq:sart}\begin{aligned}
P_n &= \frac{1}{n} P_{n-1} + \frac{n-1}{n} \cdot \frac{1}{n-1} P_{n-2} +
\frac{n-1}{n} \cdot \frac{n-2}{n-1} \cdot \frac{n-3}{n-2} \cdot
\frac{1}{n-3} P_{n-4}\\
&= \frac{P_{n-1} + P_{n-2} +P_{n-4}}{n}.\end{aligned}\end{equation}

It is easy to check that $P_1 = 1$, $P_2 = 1$, $P_3 = 2/3$, $P_4=2/3$,
$P_5 = 7/15$; in particular, $P_1\geq P_2\geq P_3\geq P_4\geq P_5$. Using
this as the base case of an induction, and applying
(\ref{eq:sart}) for the inductive step, we obtain that $P_n$ is non-increasing
for $n\geq 1$. Hence, $P_n\leq 3 P_{n-4}/n$, and so,
for all $n\geq 1$,
\[\begin{aligned}
P_n&\leq \frac{3^{\lfloor n/4\rfloor}}{n (n-4) (n-8) \dotsc (n-4 \lfloor n/4\rfloor)} < \frac{1}{\lfloor n/4\rfloor !}\\
&<
\frac{1+O(1/n)}{((n/4-1)/e)^{n/4-1} \sqrt{2 \pi (n/4-1)}} \ll
  \frac{1}{(n/4e)^{n/4-1/2}},\end{aligned}\]
where we are using Stirling's formula.

  How many maps $f:\mathbb{Z}/n\mathbb{Z}\to \mathbb{Z}/n\mathbb{Z}$ satisfy
  the condition $f(x+1) = 2 f(x)$ for $\geq (1-\epsilon) n$ values of
  $x$ in $\mathbb{Z}/n\mathbb{Z}$? Such a condition ``forces'' the value
  of $f(x+1)$ wherever it is fulfilled; hence, we have the freedom to choose
  $f(x+1)$ only at $m = \lfloor \epsilon n\rfloor$ places.
  We also have the freedom to choose where those places are.
  Hence, the number of elements of the set 
  $S_n$ of all such maps $f$ is bounded by
  \[|S_n|\leq \binom{n}{m} \cdot \frac{n!}{(n-m)!} \ll \frac{n^n}{m^m
    (n-m)^{n-m}} \cdot n^m \leq \left(\frac{1}{\epsilon^\epsilon (1-\epsilon)^{
      1-\epsilon}}\right)^n n^{\epsilon n}.\]

  Hence, for any $\epsilon<1/4$, the product $|S_{n}|\cdot P_n$ goes to
  $0$ as $n\to \infty$; indeed, it is $o(e^{- c n})$ for any $c>0$.
  This implies that
  \[\lim_{N\to \infty} \sum_{n\geq N} |S_n|\cdot P_n = 0.\]

  How to interpret this? If we model the event $f\circ f\circ f\circ f = e$
  as a random event with probability $P_n$ {\em independent} of 
  whether $f\in S_n$ (and it is here, and not elsewhere, that the argument
  becomes a heuristic, rather than a proof), then the expected value of
  the number of elements $f$ of $S_n$ satisfying $f\circ f\circ f\circ f = e$
  is $|S_n|\cdot P_n$. Hence,
  \[\sum_{n\geq N} |S_n|\cdot P_n\]
  is the expected value of the number of elements of $f$ of $S_n$,
  $n\geq N$, satisfying $f\circ f\circ f\circ f = e$, and thus is an upper
  bound on the probability that there is at least one $n\geq N$ and
  at least one $f\in S_n$ such that $f\circ f\circ f\circ f = e$.
  As we just saw, this upper bound goes to $0$ as $N\to \infty$
  (indeed, it goes to $0$ faster than any exponential).
  In other words, the conclusion of Theorem \ref{thm:jut}
  would seem to be made implausible by
  a simple probabilistic model.
    
There are two things to discuss: why this heuristic is no proof, and 
how much weight one should place on it, if any.

  What keeps us from making this heuristic into a proof is
  the difficulty in ensuring mutual independence of our random events. 
  It is easy to see that,
  if a map $f$ satisfying $f(x+1) = m f(x)$ for $\geq (1-\epsilon) n$
  values $x\in \mathbb{Z}/n\mathbb{Z}$ is taken at random, the probability
  that $f(f(f(f(x_0))))=x_0$ for a given $x_0$ is very close to $1/n$.
  Pairwise independence is not difficult either: for $x_0\ne x_1$,
  the probability that $f(f(f(f(x))))=x$ be true for $x=x_0$ or $x=x_1$
  is very close to $1/n(n-1)$ (whether $x_0$ and $x_1$ are close to each other
  or not). The problem lies in ensuring the almost-independence of $k$ such
  events for $k$ rather large. 

The strength of the heuristic, or rather its weakness, is a non-obvious matter.
In number theory,
arithmetic properties that have no reason to be correlated are
  generally believed to be independent in the limit. (Examples: the
  Hardy-Littlewood conjectures, Chowla's conjecture, and statements
  on $\sum_{n\leq N} \mu(n)$ equivalent to the Riemann hypothesis.)
This would seem to support the heuristic. However, we are on unfamiliar
terrain here: the functions we are working with behave locally like arithmetic
functions almost everywhere, but need not be close to them globally.

More importantly, in response to an earlier version of the current version,
Kassabov, Kuperberg and Riley have proven a result that goes against 
the heuristic.

\begin{theorem}[\cite{kkr}]
For every $\epsilon>0$ and every $C>0$, there is an $N$ such that, for any
$n\geq N$ and any $(\log \log n)/C < b_n < C \log n$ coprime to $n$, 
there is a bijection $f:\mathbb{Z}/n\mathbb{Z} \to \mathbb{Z}/n\mathbb{Z}$
such that 
\[f(x+1) = b_n f(x)\]
for at least $(1-\epsilon) n$ values of $x\in \mathbb{Z}/n\mathbb{Z}$, and
\[f(f(f(f(x))))=x\]
for all $x\in \mathbb{Z}/n\mathbb{Z}$.
\end{theorem}

This theorem does not necessarily mean that the functions $f$ in Theorem \ref{thm:jut} are likely to exist, or that $H_{4,m}$ is likely to be sofic. However,
it does mean that the na\"{\i}ve heuristic examined in this section should
be distrusted.

Moreover, as we said, the work of Glebsky \cite{Gleb} implies the following.
\begin{corollary}[to \cite{Gleb}]
  For every $m>2$ and every $\epsilon>0$, there is an $N$ such that, for any
$n\geq N$ coprime to $m$, 
there is a bijection $f:\mathbb{Z}/n\mathbb{Z} \to \mathbb{Z}/n\mathbb{Z}$
such that 
\[f(x+1) = m f(x)\]
for at least $(1-\epsilon) n$ values of $x\in \mathbb{Z}/n\mathbb{Z}$, and
\[f(f(f(f(x))))=x\]
for all $x\in \mathbb{Z}/n\mathbb{Z}$.
\end{corollary}
While Glebsky works with $n = p^k$, a statement for such $n$ implies a
statement for all large enough $n$: just choose $N$ large enough
that, for every $n\geq N$, there is a multiple of $p^k$ between
$(1-\epsilon/2) n$ and $n$. Then proceed as we did right after
(\ref{eq:hudd}).



\begin{center}
* * *
\end{center}

  Incidentally, from the perspective of a heuristic model such as the one we
  have just discussed, it is wholly unsurprising that Prop.~\ref{prop:norvi}
  gives only that the number of points at which $f(x+1)\ne m f(x)$
  or $f(f(f(f(x))))$ is at least $m^{r/4-1} = n^{1/4}/m$.
  The proof of Prop.~\ref{prop:norvi} bounds, in effect, the number of $f$
  satisfying a much looser condition than $f(x+1) = m f(x)$ outside 
  $m$ points: it allows $f(x)$ to be $c(x) m^x$, where $c(x)$ can be any of
  $m$ distinct values. It is easy to see that there are very roughly
  $n^m m^n$ maps
  satisfying such a condition; this is larger than $1/P_n$ whenever
  $m\geq n^{1/4}$.
  
\section{Further questions}\label{sec:furthq}
\subsection{Bounds for stability}\label{subs:bounsta}
We have chosen to use, for the most part, finitary versions of existing
tools to solve problems with finitary formulations. This is not the only
option: for example, the work of Arzhantseva and P\u{a}unescu
\cite{MR3350728}
applies
\cite{MR2823074} (which uses ultraproducts heavily) to prove finitary
statements.



Arzhantseva and P\u{a}unescu
undertook a general study of the following question: given two (or more)
elements of $\Sym(n)$ that almost satisfy a relation, must they be
close to elements that actually satisfy it?
To take a particular case --  they showed
that the answer is ``yes'' ({\em stability}) for the commutator relation
\[\lbrack x,y\rbrack = e,\]
or, for that matter, for the set of all commutator relations
$\lbrack x_i,x_j\rbrack = e$ ($1\leq i,j\leq k$) among $k$ elements $x_i$.
(This is \cite[Main Thm.]{MR3350728}.)

A line of further inquiry suggests itself: how close is close? That is --
the stability of the commutator means that for every $\epsilon>0$ there is a
$\delta>0$ such that, if $d_h(\lbrack x,y\rbrack,e)<\delta$, then there
are $x'$, $y'$ such that $d_h(x,x'),d_h(y,y')<\epsilon$ and $\lbrack x',y'\rbrack = e$; the question is -- how does $\delta$ depend on $\epsilon$?

While the method based on \cite{MR3068400}\footnote{We note that
  \cite[\S 4]{MR3068400} credits \cite{MR910005} as the source of some of
  the main ideas
  in the method.}
that we follow here, being finitary and explicit, can in
principle be made to give a bound in answer to this question,
it seems likely that such a bound would be far from optimal.
It has been pointed out to us by E. Hrushovski (in a different context) that
methods based on ultraproducts can at least sometimes give computable
bounds for some problems; still, it seems very likely that the usage of such
methods here would give even worse bounds.

\subsection{Sofic profile and sofic dimension growth}

Cornulier \cite{MR3160544} defines the {\em sofic profile} of a group
$G$ in terms of the growth as $\delta\to 0$, for $S$ fixed, of the least
$n$ such that $G$ has an $(S,\delta,n)$-sofic approximation.
In particular, a group $G$ has {\em at most linear sofic profile} if
for every finite $S$ there is a constant $c_S>0$ such that,
for every $\delta>0$, there is an $n\leq c_S/\delta$ such that
$G$ has an $(S,\delta,n)$-sofic approximation.

The definition of sofic profile was preceded by that of {\em
  sofic dimension growth} \cite{ArzhCher}. For $S$ fixed,
define $\phi_S(r)$ to be the least $n$ such that $G$ has
an $(S^r,1/r,n)$-sofic approximation. Then the question is the
growth of $\phi_S(r)$ as $r\to \infty$.

Strangely enough, it is unknown whether it is the case that all groups
have at most linear sofic profile; that question is
\cite[Problem 3.18]{MR3160544}. We know that there are groups that
have linear sofic profile (meaning: the least $n$ such that
$G$ has an $(S,\delta,n)$-sofic approximation obeys
$c_S'/\delta \leq n \leq c_S/\delta$ for some $c_S,c_S'>0$). As
\cite{MR3160544} shows, a group that has at most linear sofic profile either has
linear sofic profile or is a {\em LEF group}. The concept of a LEF group was
introduced by Gordon and Vershik \cite{MR1458419}.
In brief, a group is LEF if one can set $\delta=0$;
for finitely presented groups, being LEF is equivalent to being residually
finite (to see this, choose $S$ large enough to include all relations).

Since the Higman group has no proper subgroups of finite index 
\cite{MR0038348}, it is not residually finite, and thus it is not LEF.
Does the Higman group have linear sofic profile? It is natural
to venture that it does not; the question remains tantalizingly
open.

It does not seem viable to address this question with
the tools in \S \ref{sec:amensof}, since they worsen dramatically the
dependence of $n$ on $\delta$. It is not clear whether one can proceed
similarly with other tools. The goal would be to show that, if the
Higman group has linear sofic profile, then, for every $n_0\geq 1$, there is
an $n$ with $|n-n_0| = O(1)$ such that
\begin{equation}\label{eq:lisar}f_{2,n}(f_{2,n}(f_{2,n}(f_{2,n}(x))))=x\end{equation}
holds for all but $O(1)$ values of $x\in \{0,1,\dotsc,n-1\}$; here
$f_{m,n}$ is as in Def.~\ref{defn:fg}. The converse is clear: if there are enough
$n$ for which (\ref{eq:lisar}) holds for all but $O(1)$ values
of $x\in \{0,1,\dotsc,n-1\}$, then the Higman group does have linear sofic
profile.

As we said in the introduction,
we know that (\ref{eq:lisar}) has few solutions for some specific $n$, but
such $n$ are rare (fifth and higher powers of primes) \cite[Cor. 3]{MR3118384},
\cite[Thm. 5.7]{MR2999154}.
It does seem very unlikely that there are infinitely many $n$ for which
(\ref{eq:lisar}) holds for all but $O(1)$ values of $x\in \{0,1,\dotsc,n-1\}$,
but showing that this is the case is an open problem.

\bibliographystyle{alpha}
\bibliography{higman}

\begin{thebibliography}{Kon79b}

\bibitem[AC]{ArzhCher}
Goulnara Arzhantseva and Pierre-Alain Cherix.
\newblock Quantifying metric approximations of groups.
\newblock Unpublished.

\bibitem[AP15]{MR3350728}
Goulnara Arzhantseva and Liviu P{\u{a}}unescu.
\newblock Almost commuting permutations are near commuting permutations.
\newblock {\em J. Funct. Anal.}, 269(3):745--757, 2015.

\bibitem[Cor11]{MR2794910}
Yves Cornulier.
\newblock A sofic group away from amenable groups.
\newblock {\em Math. Ann.}, 350(2):269--275, 2011.

\bibitem[Cor13]{MR3160544}
Yves Cornulier.
\newblock Sofic profile and computability of {C}remona groups.
\newblock {\em Michigan Math. J.}, 62(4):823--841, 2013.

\bibitem[DKP14]{MR3130315}
Ken Dykema, David Kerr, and Mika{\"e}l Pichot.
\newblock Sofic dimension for discrete measured groupoids.
\newblock {\em Trans. Amer. Math. Soc.}, 366(2):707--748, 2014.

\bibitem[ES06]{MR2220572}
G{\'a}bor Elek and Endre Szab{\'o}.
\newblock On sofic groups.
\newblock {\em J. Group Theory}, 9(2):161--171, 2006.

\bibitem[ES11]{MR2823074}
G{\'a}bor Elek and Endre Szab{\'o}.
\newblock Sofic representations of amenable groups.
\newblock {\em Proc. Amer. Math. Soc.}, 139(12):4285--4291, 2011.

\bibitem[Gle]{Gleb}
Lev Glebsky.
\newblock $p$-quotients of the {G.} {H}igman group.
\newblock Preprint. Available at \url{arxiv.org/abs/1604.06359}.

\bibitem[Gle13]{MR3118384}
Lev Glebsky.
\newblock Cycles in repeated exponentiation modulo {$p^n$}.
\newblock {\em Integers}, 13, 2013.
\newblock Paper No. A66, 7 pages.

\bibitem[GR03]{MR1982971}
Oded Goldreich and Vered Rosen.
\newblock On the security of modular exponentiation with application to the
  construction of pseudorandom generators.
\newblock {\em J. Cryptology}, 16(2):71--93, 2003.

\bibitem[GS10]{MR2644246}
Lev Glebsky and Igor~E. Shparlinski.
\newblock Short cycles in repeated exponentiation modulo a prime.
\newblock {\em Des. Codes Cryptogr.}, 56(1):35--42, 2010.

\bibitem[Hig51]{MR0038348}
Graham Higman.
\newblock A finitely generated infinite simple group.
\newblock {\em J. London Math. Soc.}, 26:61--64, 1951.

\bibitem[HR12]{MR2999154}
Joshua Holden and Margaret~M. Robinson.
\newblock Counting fixed points, two-cycles, and collisions of the discrete
  exponential function using {$p$}-adic methods.
\newblock {\em J. Aust. Math. Soc.}, 92(2):163--178, 2012.

\bibitem[KKR]{kkr}
Vivian Kuperberg, Martin Kassabov, and Tim Riley.
\newblock Soficity and variations on higman's group.
\newblock Preprint. Based in part on V. Kuperberg's senior thesis, under the
  direction of T. Riley.

\bibitem[KL13]{MR3068400}
David Kerr and Hanfeng Li.
\newblock Soficity, amenability, and dynamical entropy.
\newblock {\em Amer. J. Math.}, 135(3):721--761, 2013.

\bibitem[Kon79a]{MR548524}
Sergei~V. Konjagin.
\newblock Letter to the editors: ``{T}he number of solutions of congruences of
  the {$n$}th degree with one unknown''.
\newblock {\em Mat. Sb. (N.S.)}, 110(152)(1):158, 1979.

\bibitem[Kon79b]{MR542556}
Sergei~V. Konjagin.
\newblock The number of solutions of congruences of the {$n$}th degree with one
  unknown.
\newblock {\em Mat. Sb. (N.S.)}, 109(151)(2):171--187, 327, 1979.

\bibitem[OW87]{MR910005}
Donald~S. Ornstein and Benjamin Weiss.
\newblock Entropy and isomorphism theorems for actions of amenable groups.
\newblock {\em J. Analyse Math.}, 48:1--141, 1987.

\bibitem[Pes08]{MR2460675}
Vladimir~G. Pestov.
\newblock Hyperlinear and sofic groups: a brief guide.
\newblock {\em Bull. Symbolic Logic}, 14(4):449--480, 2008.

\bibitem[PS98]{MR1670959}
Sarvar Patel and Ganapathy~S. Sundaram.
\newblock An efficient discrete log pseudo-random generator.
\newblock In {\em Advances in cryptology---{CRYPTO} '98 ({S}anta {B}arbara,
  {CA}, 1998)}, volume 1462 of {\em Lecture Notes in Comput. Sci.}, pages
  304--317. Springer, Berlin, 1998.

\bibitem[Sch71]{MR0291298}
Paul~E. Schupp.
\newblock Small cancellation theory over free products with amalgamation.
\newblock {\em Math. Ann.}, 193:255--264, 1971.

\bibitem[Tho12]{MR2900231}
Andreas Thom.
\newblock About the metric approximation of {H}igman's group.
\newblock {\em J. Group Theory}, 15(2):301--310, 2012.

\bibitem[VG97]{MR1458419}
Anatoly~M. Vershik and E.~I. Gordon.
\newblock Groups that are locally embeddable in the class of finite groups.
\newblock {\em Algebra i Analiz}, 9(1):71--97, 1997.

\end{thebibliography}
\end{document}